\documentclass[a4paper]{article}
\usepackage{geometry}
\usepackage{amsmath,amssymb,amsfonts,amsthm}
\usepackage{latexsym}
\usepackage{indentfirst}
\usepackage{graphicx}
\usepackage[colorlinks=true]{hyperref}
\usepackage{mathrsfs} 
\usepackage{chngcntr}
\usepackage{color}
\usepackage{tikz}

\newtheorem{Thm}{Theorem}[section]

\newtheorem{Lem}{Lemma}[section]
\newtheorem{Pro}{Proposition}[section]

\newtheorem{Rek}{Remark}[section]
\newtheorem{Def}{Definition}[section]

\numberwithin{equation}{section}

\begin{document}

\title{Spectral properties and bound states of the Dirac equation on periodic quantum graphs}
\author{
  Zhipeng Yang\thanks{Email:yangzhipeng326@163.com}\quad and\quad Ling Zhu\\
  \small Department of Mathematics, Yunnan Normal University, Kunming, China.\\
  \small Yunnan Key Laboratory of Modern Analytical Mathematics and Applications, Kunming, China.\\
}
\date{}
\maketitle

\begin{abstract}
We investigate nonlinear Dirac equations on a periodic quantum graph $\mathcal G$ and develop a variational approach to the existence and multiplicity of bound states. After introducing the Dirac operator on $\mathcal G$ with a $\mathbb Z^{d}$-periodic potential, we describe its spectral decomposition and work in the natural energy space. Under asymptotically linear or superquadratic assumptions on the nonlinearity, we establish the required linking geometry and a Cerami-type compactness property modulo $\mathbb Z^{d}$-translations. As a consequence, we prove the existence of at least one bound state and, when the nonlinearity is even, infinitely many geometrically distinct bound states.
\end{abstract}

\paragraph{Keywords.} Nonlinear Dirac equation, Quantum graph, Variational methods.

\paragraph{2010 AMS Subject Classification.} 35R02; 35Q41; 81Q35.

\section{Introduction}
We are concerned with the existence and multiplicity of stationary states for the nonlinear Dirac equation (NLDE)
\begin{equation}\label{eq-1.1}
	i\hbar \partial_t\psi
	= ic\hbar\sigma_{1}\frac{d}{dx}\psi - mc^{2}\sigma_{3}\psi + G_{\psi}(x,\psi),
\end{equation}
for a function $\psi: \mathbb{R} \times \mathcal{G} \rightarrow \mathbb{C}^{2}$ representing the wave function of the state of an electron on a periodic quantum graph $\mathcal{G}$.
Here $G: \mathcal{G} \times \mathbb{C}^{2} \rightarrow \mathbb{R}$ is a nonlinear potential, $c$ denotes the speed of light, $m>0$ the mass of the electron, $\hbar$ is Planck's constant, and $\sigma_{1}$ and $\sigma_{3}$ are the Pauli matrices
\[
\sigma_{1}=\begin{pmatrix}
	0 & 1\\
	1 & 0\\
\end{pmatrix},
\qquad
\sigma_{3}=\begin{pmatrix}
	1 & 0\\
	0 & -1\\
\end{pmatrix}.
\]
Throughout, $G_{\psi}(x,\psi)$ denotes the real Fr\'echet derivative of the map $\psi\mapsto G(x,\psi)$ after identifying $\mathbb{C}^{2}\cong\mathbb{R}^{4}$.

Assuming that $G$ satisfies $G(x, e^{i \theta} \psi)=G(x, \psi)$ for all $\theta \in \mathbb{R}$, we look for solutions of the form
\[
\psi(t, x)=e^{-i \vartheta t} u(x).
\]
These are the stationary solutions of \eqref{eq-1.1}. The function $u: \mathcal{G} \rightarrow \mathbb{C}^{2}$ then solves
\begin{equation}\label{eq-1.2}
	-ic\hbar \sigma_{1}\frac{d}{dx}u + mc^{2}\sigma_{3}u = G_{u}(x,u)-\hbar \vartheta u,
\end{equation}
where $G_u(x,u):=G_{\psi}(x,u)$.

Dividing \eqref{eq-1.2} by $\hbar c$ and setting
\[
a=\frac{mc^{2}}{\hbar c}>0,\quad \omega=\frac{\vartheta}{c}\in\mathbb{R},\quad
F(x,u)=\frac{1}{\hbar c}G(x,u),
\]
we are led to study
\begin{equation}\label{eq-1.3}
	-i\sigma_{1}\frac{d}{dx}u + a\sigma_{3}u + \omega u = F_{u}(x,u),
\end{equation}
where $a>0$ and $\omega \in \mathbb{R}$, and $F_u$ is understood as the real gradient on $\mathbb{C}^2\cong\mathbb{R}^4$. 
	
	In recent years, linear and nonlinear Dirac equations have attracted considerable attention,
	since they arise as effective models in many physical contexts. The nonlinear Dirac equation was initially introduced as a field equation for relativistic interacting fermions and
	was later used to describe quark confinement in particle physics. In the setting of
	Bose--Einstein condensates, the NLDE has been analyzed, for instance, in \cite{MR2542748},
	and relativistic soliton solutions of nonlinear Dirac equations for Bose--Einstein condensates
	in cellular optical lattices were studied in \cite{MR3371128}.
	
	Several works have also investigated Dirac type models in periodic media and lattice
	structures. In \cite{MR2947949}, honeycomb lattice potentials and the occurrence of Dirac
	points were analyzed. The evolution of wave packets with spectral concentration near such
	Dirac points was studied in \cite{MR3162492}, while \cite{MR3749319} derived effective
	nonlinear Dirac type equations via multiscale expansions together with rigorous error
	estimates. Moreover, the existence of steady state solutions for cubic and Hartree type
	Dirac equations describing honeycomb structures and graphene samples has been established
	in \cite{MR3705703,MR3794033,MR3859463}.
	
	The study of Dirac operators on quantum graphs has recently attracted increasing interest.
	In \cite{MR3557323}, the Dirac--Krein operator with a potential on a compact quantum graph
	with finitely many edges was investigated, yielding Krein type resolvent formulas, results
	on the multiplicity of eigenvalues, and interlacing properties of the spectrum. Dirac
	operators on graphs and associated trace formulas for their spectra were derived in
	\cite{MR1965289}. In \cite{MR2459887}, the spectra of Laplace and Dirac operators on
	quantum graphs were analyzed using boundary triplet theory. In the simplified situation
	of a three-edge star graph, a nonlinear Dirac equation on a network was considered in
	\cite{MR3871194}, where stationary soliton solutions on a simple quantum graph were
	constructed.
	
	An overview of existence results for one-dimensional NLDE is as follows.
	Cacciapuoti et al.~\cite{MR3665199} studied the Cauchy problem for a one-dimensional
	nonlinear Dirac equation with nonlinearities concentrated at a single point, proving global
	well-posedness and conservation of mass and energy. Pelinovsky \cite{MR2883845} considered one-dimensional nonlinear Dirac equations on the
	line and established global existence of small-norm solutions for very general classes of
	equations with cubic and higher-order nonlinear terms. More recently, NLDE on noncompact
	quantum graphs has been investigated in \cite{GuLiRuzhanskyYang2025,GuRuzhanskyWeiYang2025},
	where the nonrelativistic limit and qualitative properties of bound states, as well as
	multiplicity and semiclassical concentration of solutions for equations with external
	potentials, have been obtained.

	In this article, we assume that $\mathcal{G}$ is a connected metric graph which carries
	a free, cocompact action of the group $\mathbb{Z}^{d}$ by graph automorphisms.
	More precisely, there exists a group homomorphism
	\[
	\mathbb{Z}^{d}\ni k\longmapsto T^{k}\in\operatorname{Aut}(\mathcal{G})
	\]
	such that:
	\begin{itemize}
		\item[$(a)$] the action $\mathbb{Z}^{d}\times\mathcal{G}\to\mathcal{G}$,
		\((k,x)\mapsto T^{k}(x)\), is free and by isometries on each edge;
		\item[$(b)$] the quotient graph $\mathcal{G}/\mathbb{Z}^{d}$ is compact
		(a finite metric graph);
		\item[$(c)$] there exists a compact connected subgraph
		$\mathcal{K}\subset\mathcal{G}$ such that
		\[
		\mathcal{G} = \bigcup_{k\in\mathbb{Z}^{d}} T^{k}(\mathcal{K}),
		\]
		and $T^{k}(\mathcal{K})\cap T^{\ell}(\mathcal{K})$ consists only of boundary
		vertices whenever $k\neq\ell$.
	\end{itemize}
	The set $\mathcal{K}$ is called a fundamental cell of the periodic graph
	$\mathcal{G}$. We study the following more general Dirac equation on the periodic quantum graph:
	\begin{equation}\label{eq-1.5}
		-i\sigma_{1}\frac{d}{dx}u + \big(V(x)+a\big)\sigma_{3}u+\omega u=F_{u}(x,u).
	\end{equation}
    	
	For our purposes, we assume:
	\begin{itemize}
		\item[(\(\omega\))] \(\omega\in(-a,a)\).
		
		\item[(\(F_{0}\))] \(F\in C^{1}\big(\mathcal{G}\times\mathbb{C}^{2},[0,\infty)\big)\).
		
		\item[(\(F_{1}\))] \(F\) is invariant under the \(\mathbb{Z}^{d}\)-action, that is
		\[
		F(T^{k}x,u)=F(x,u)\quad\text{for all }k\in\mathbb{Z}^{d},\ x\in\mathcal{G},\ u\in\mathbb{C}^{2}.
		\]
		
		\item[(\(F_{2}\))] \(F_{u}(x,u)=o(|u|)\) as \(u\to 0\), uniformly in \(x\in\mathcal{G}\).
	\end{itemize}
	
	We first treat the case where $F_u(x,u)$ is asymptotically linear in $u$, while the associated quantity $\hat F$ has a positive but possibly subquadratic growth at infinity, in the sense of $(F_3)$–$(F_4)$.
	Set
	\[
	\omega_{0}=\min\{a+\omega,a-\omega\}
	\quad\text{and}\quad
	\hat{F}(x,u)=\frac{1}{2}\,F_{u}(x,u)\cdot u-F(x,u),
	\]
	where \(\cdot\) denotes the real scalar product on \(\mathbb{C}^{2}\cong\mathbb{R}^{4}\).
	We require:
	
	\begin{itemize}
		\item[(\(F_{3}\))] There exists \(b\in C^{1}(\mathcal{G},\mathbb{R})\) such that
		\[
		\frac{\big|F_{u}(x,u)-b(x)u\big|}{|u|}\to 0
		\quad\text{as }|u|\to\infty
		\]
		uniformly in \(x\in\mathcal{G}\), and
		\[
		\inf_{\mathcal{G}} b > \sup_{\mathcal{G}} V +a+ \omega.
		\]
		
		\item[(\(F_{4}\))] \(\hat{F}(x,u)>0\) if \(u\neq 0\), and there exist \(\kappa\in(0,2)\), \(R>0\) and \(c_{1}>0\) such that
		\[
		\hat{F}(x,u)\ge c_{1}|u|^{\kappa}
		\quad\text{whenever }|u|\ge R.
		\]
		
		\item[(\(F_{5}\))] \(F\in C^{2}\big(\mathcal{G}\times\mathbb{C}^{2},[0,\infty)\big)\), and there exist \(\nu\in(0,1)\) and \(C_{1}>0\) such that
		\[
		|F_{uu}(x,u)|\le C_{1}\big(1+|u|^{\nu}\big)
		\quad\text{for all }(x,u)\in\mathcal{G}\times\mathbb{C}^{2}.
		\]
	\end{itemize}
	
	First, we consider equations with periodic potentials:
	
	\begin{itemize}
		\item[(V$_1$)] $V\in C^1(\mathcal{G},[0,\infty))$ is bounded and
		$\mathbb Z^d$–periodic, that is,
		\[
		V(T^k x)=V(x)\quad\forall\,k\in\mathbb Z^d,\ x\in\mathcal G.
		\]
	\end{itemize}
	
	We obtain the following result.
	
	\begin{Thm}\label{theo-1.1}
		Let $\mathcal{G}$ be a periodic quantum graph as above with fundamental cell $\mathcal{K}$, and assume that $(\omega)$, $(V_{1})$ and $(F_{0})$--$(F_{5})$ hold. Then the NLDE \eqref{eq-1.5} admits at least one bound state $u$. If, in addition to the above assumptions, $F$ is even in $u$, then the NLDE \eqref{eq-1.5} possesses infinitely many geometrically distinct bound states.
	\end{Thm}
	
	Next we consider the superquadratic case, where we assume:
	
	\begin{itemize}
		\item[(\(F_{6}\))] \(F(x,u)\,|u|^{-2}\to\infty\) as \(|u|\to\infty\), uniformly in \(x\in\mathcal{G}\).
		
		\item[(\(F_{7}\))] \(\hat{F}(x,u)>0\) if \(u\neq 0\), and there exist \(\sigma>1\), \(r>0\) and \(c_{2},c_{3}>0\) such that
		\begin{itemize}
			\item[(i)] \(\hat{F}(x,u)\ge c_{2}|u|^{2}\) if \(|u|\ge r\),
			\item[(ii)] \(|F_{u}(x,u)|^{\sigma}\le c_{3}\,\hat{F}(x,u)\,|u|^{\sigma}\) if \(|u|\ge r\).
		\end{itemize}
	\end{itemize}
	
	We then have the following theorem.
	
	\begin{Thm}\label{theo-1.2}
		Let $\mathcal{G}$ be a periodic quantum graph as above with fundamental cell $\mathcal{K}$, and assume that $(\omega)$, $(V_{1})$ and $(F_{0})$--$(F_{2})$, $(F_{5})$, $(F_{6})$, $(F_{7})$ hold. Then the NLDE \eqref{eq-1.5} admits at least one bound state $u$. If, in addition to the above assumptions, $F$ is even in $u$, then the NLDE \eqref{eq-1.5} possesses infinitely many geometrically distinct bound states.
	\end{Thm}

	\begin{Rek}\label{rem-1.1}
In this paper, two bound states $u_{1}$ and $u_{2}$ are said to be
geometrically distinct if they are not related by the natural
symmetries of the problem. Under the periodicity assumptions $(V_1)$ and $(F_1)$ and the gauge invariance
$F(x,e^{i\theta}u)=F(x,u)$ for all $\theta\in\mathbb{R}$, the equation is invariant under the phase action and the $\mathbb{Z}^{d}$-action on $\mathcal{G}$.

More precisely, for each $k\in\mathbb{Z}^{d}$ we define the translated
spinor $k*u$ by
\[
(k*u)(x)=u\big(T^{-k}x\big),\qquad x\in\mathcal{G},
\]
where $T^{k}\in\mathrm{Aut}(\mathcal{G})$ is the graph automorphism
introduced above. Then $u_{1}$ and $u_{2}$ are geometrically distinct
if
\[
u_{2}\neq e^{i\theta}\,(k*u_{1})
\quad\text{for all }\theta\in\mathbb{R},\ k\in\mathbb{Z}^{d}.
\]
\end{Rek}

\par
A distinctive feature of the NLDE is the strong indefiniteness of its action functional. Unlike in the Schr\"odinger case, where the energy functional is typically bounded from below \cite{MR1400007}, the Dirac action lacks coercivity in any natural function space, since the spectrum of the underlying operator is unbounded both above and below. This indefiniteness precludes the use of direct minimization techniques and calls for more sophisticated critical point theories \cite{MR2389415}. Moreover, the noncompactness of the graph undermines the Palais--Smale (PS) condition, a cornerstone of classical variational methods. Consequently, standard approaches to existence and multiplicity---such as the mountain pass theorem or symmetric minimax principles---require substantial adaptation or replacement.
\par
In this work, we combine several strategies to overcome these difficulties. First, we exploit the spectral decomposition of the self-adjoint Dirac operator to split the form domain into positive and negative subspaces, which yields a natural strongly indefinite variational framework. In the periodic setting, compactness may fail due to translation invariance on $\mathcal{G}$; this is handled by a concentration--compactness argument modulo the $\mathbb{Z}^{d}$-action, which allows one to recover compactness for suitable Palais--Smale sequences up to translations. Next, we use variational and symmetric critical point arguments to construct critical points, taking into account the vertex conditions at the graph junctions.
\par
The paper is organized as follows. In Section~2 we introduce the quantum graph model, specify the vertex conditions, and formulate the NLDE as a variational problem, together with the abstract critical point theorem that we shall use. Section~3 is devoted to the case of periodic potentials: we study the spectral properties of the Dirac operator under assumption $(V_{1})$ and prove Theorems~\ref{theo-1.1} and~\ref{theo-1.2}.

	\section{Preliminaries}
	\subsection{Quantum graphs and functional setting}
	
	We refer to \cite{MR4438617} and the references therein for a general introduction to quantum graphs, and we briefly recall here the basic notions needed in this paper.
	
	A quantum graph $\mathcal{G}=(V,E)$ is a connected metric graph consisting of a countable set of edges and vertices, with each vertex of finite degree, possibly with multiple edges and self-loops. Each edge is identified either with a finite interval or with a half-line, and the edges are glued together at their endpoints according to the topology encoded by the underlying combinatorial graph.
	
	Unbounded edges are identified with $[0,+\infty)$ and are called half-lines, while bounded edges are identified with closed intervals $I_{e}=[0,\ell_{e}]$, where $\ell_{e}>0$ denotes the length of the edge $e$. We assign to each edge $e$ a coordinate $x_{e}\in I_{e}$ measuring the arclength along $e$. If the number of edges and vertices is finite and the graph does not contain any half-lines, then $\mathcal{G}$ is compact. In general, and in particular in the periodic setting considered in this paper, $\mathcal{G}$ will be noncompact.

    In the present paper, $\mathcal{G}$ is a periodic quantum graph as described in the Introduction, and we fix a compact fundamental cell $\mathcal{K}\subset\mathcal{G}$ such that
\[
\mathcal{G} = \bigcup_{k\in\mathbb{Z}^{d}} T^{k}(\mathcal{K}),
\]
where $T^{k}$ is a free, cocompact $\mathbb{Z}^{d}$-action by graph automorphisms. The cell $\mathcal K$ will be used as a fundamental domain for the periodic structure of $\mathcal G$.
	
A function $u:\mathcal G\to\mathbb C$ can be identified with a family
$(u_{e})_{e\in\mathcal E}$, where $u_{e}:I_{e}\to\mathbb C$ is the
restriction of $u$ to the edge $I_{e}$. For $1\le p<\infty$ we consider
\[
L^{p}(\mathcal G)
=\Bigl\{u=(u_{e})_{e\in\mathcal E}:\ u_{e}\in L^{p}(I_{e})\ \text{for all }e,\
\sum_{e\in\mathcal E}\|u_{e}\|_{L^{p}(I_{e})}^{p}<\infty\Bigr\},
\]
with norm
\[
\|u\|_{L^{p}(\mathcal G)}^{p}
=\sum_{e\in\mathcal E}\|u_{e}\|_{L^{p}(I_{e})}^{p},
\qquad 1\le p<\infty,
\]
and
\[
\|u\|_{L^{\infty}(\mathcal G)}
=\sup_{e\in\mathcal E}\|u_{e}\|_{L^{\infty}(I_{e})}.
\]
We define
\[
H^{1}(\mathcal G)
=\Bigl\{u=(u_{e})_{e\in\mathcal E}:\ u_{e}\in H^{1}(I_{e})\ \text{for all }e,\
\sum_{e\in\mathcal E}\|u_{e}\|_{H^{1}(I_{e})}^{2}<\infty\Bigr\},
\]
with norm
\[
\|u\|_{H^{1}(\mathcal G)}^{2}
=\|u'\|_{L^{2}(\mathcal G)}^{2} + \|u\|_{L^{2}(\mathcal G)}^{2},
\]
where $u'=(u_{e}')_{e\in\mathcal E}$ denotes the family of weak derivatives
along each edge.

A spinor $u:\mathcal G\to\mathbb C^{2}$ is a map $u=(u^{1},u^{2})^{T}$ whose
components $u^{1},u^{2}:\mathcal G\to\mathbb C$ are scalar functions. Equivalently,
one may regard $u$ as a family of $2$-spinors
\[
u_{e} = \binom{u_{e}^{1}}{u_{e}^{2}} : I_{e}\to\mathbb C^{2},
\qquad e\in\mathcal E.
\]
We consider
\[
L^{p}(\mathcal G,\mathbb C^{2})
=\{u=(u^{1},u^{2})^{T}:\ u^{1},u^{2}\in L^{p}(\mathcal G)\},
\]
with
\[
\|u\|_{L^{p}(\mathcal G,\mathbb C^{2})}^{p}
=\|u^{1}\|_{L^{p}(\mathcal G)}^{p}
+\|u^{2}\|_{L^{p}(\mathcal G)}^{p},
\quad 1\le p<\infty,
\]
and
\[
\|u\|_{L^{\infty}(\mathcal G,\mathbb C^{2})}
=\max\bigl\{\|u^{1}\|_{L^{\infty}(\mathcal G)},
\|u^{2}\|_{L^{\infty}(\mathcal G)}\bigr\}.
\]
Similarly,
\[
H^{1}(\mathcal G,\mathbb C^{2})
=\{u=(u^{1},u^{2})^{T}:\ u^{1},u^{2}\in H^{1}(\mathcal G)\},
\]
with
\[
\|u\|_{H^{1}(\mathcal G,\mathbb C^{2})}^{2}
=\|u^{1}\|_{H^{1}(\mathcal G)}^{2}
+\|u^{2}\|_{H^{1}(\mathcal G)}^{2}.
\]

Since each edge $I_e$ is isometric either to a bounded interval with $\ell_- \le \ell_e \le \ell_+$
or to a half-line, the one-dimensional Sobolev embedding yields a constant $C_S>0$,
independent of $e$, such that
\[
|u_e|_{L^\infty(I_e)} \le C_S \|u_e\|_{H^1(I_e)} \qquad \forall\,e\in E.
\]
Taking the supremum over $e\in E$ we obtain
\[
|u|_{L^{\infty}(\mathcal{G})}
\le C_{\infty}\,\|u\|_{H^{1}(\mathcal{G})}
\]
for a suitable constant $C_\infty>0$.
	Properties of Sobolev spaces on graphs are often expressed in terms of functional inequalities. We recall the following Gagliardo--Nirenberg inequalities on graphs (see, e.g., \cite{MR3494248,MR3456809}).
	Let $\mathcal{G}$ be a noncompact metric graph with finite vertex degrees and uniformly positive lower bound on edge lengths.
If $q\in[2,+\infty)$ and $p\in[q,+\infty]$ and
\[
\alpha = \frac{2}{2+q}\Big(1-\frac{q}{p}\Big),
\]
then there exists a constant $C>0$ such that
\[
|u|_{L^{p}(\mathcal{G})}
\le C\,|u'|_{L^{2}(\mathcal{G})}^{\alpha}\,|u|_{L^{q}(\mathcal{G})}^{1-\alpha}
\quad\forall\,u\in H^{1}(\mathcal{G}).
\]
Moreover, on any graph one has the weaker inequality
\[
|u|_{L^{p}(\mathcal{G})}
\le C\,\|u\|_{H^{1}(\mathcal{G})}^{\alpha}\,|u|_{L^{q}(\mathcal{G})}^{1-\alpha},
\]
which holds for all $u\in H^{1}(\mathcal{G})$ whenever $q\in[2,+\infty)$, $p\in[q,+\infty]$ and $\alpha$ as above.

\begin{Rek}\label{rem2.1}
In this paper we use the globally continuous Sobolev space $H^{1}(\mathcal{G})$ for scalar functions, namely all edgewise components incident at a vertex take the same value there. For spinors we will specify vertex conditions explicitly when defining operators.
\end{Rek}
	
	As in the Schr\"odinger case for the Laplacian, we need to choose appropriate vertex
	conditions in order to make the Dirac operator self-adjoint on
	$L^{2}(\mathcal{G},\mathbb{C}^{2})$. In our arguments, we consider Kirchhoff-type
	vertex conditions (introduced in \cite{MR3871194}), which model the free case for
	the Dirac operator. For more details on self-adjoint extensions of Dirac operators
	on quantum graphs, we refer to \cite{MR1050469,MR2459887,yang2025bound}. We also mention
	\cite{MR3194518}, where boundary conditions for one-dimensional Dirac operators
	are studied in a model of quantum wires.

    \begin{Def}\label{def2.2}
Let $\mathcal{G}$ be a quantum graph, let $a>0$, and let $V\in L^\infty(\mathcal{G},\mathbb{R})$.
We call the Dirac operator with Kirchhoff-type vertex conditions the (unbounded) operator
\[
\mathcal{D}:\operatorname{dom}(\mathcal{D})\subset L^{2}(\mathcal{G},\mathbb{C}^{2})\to L^{2}(\mathcal{G},\mathbb{C}^{2})
\]
with action on each edge $e\in E$ given by
\begin{equation}\label{eq-2.1}
\mathcal{D}_{e}u_{e}
= -i\sigma_{1}u_{e}' + \big(a+V(x)\big)\sigma_{3}u_{e},
\qquad x\in I_{e},
\end{equation}
and with domain
\begin{equation}\label{eq-2.2}
\operatorname{dom}(\mathcal{D})
:= \Big\{ u = (u_{e})_{e\in E} : u_{e}\in H^{1}(I_{e},\mathbb{C}^{2})\ \forall e\in E,\
u \text{ satisfies } \eqref{eq-2.3} \text{ and } \eqref{eq-2.4} \Big\}.
\end{equation}
For each vertex $v\in V$, let
\[
\mathcal{E}_v=\{(e,\xi): e\in E,\ \xi\in\{0,\ell_e\},\ \text{the endpoint }x_e=\xi \text{ corresponds to } v\}.
\]
For $(e,\xi)\in\mathcal{E}_v$ we write $u_e^j(\xi)$ for the trace of the $j$-th component at the endpoint $\xi$.
The vertex conditions are
\begin{gather}\label{eq-2.3}
u_e^{1}(\xi)=u_f^{1}(\eta)
\quad \forall\, (e,\xi),(f,\eta)\in \mathcal{E}_v,\ \forall v\in V,\\
\label{eq-2.4}
\sum_{(e,\xi)\in\mathcal{E}_v} s(\xi)\,u_{e}^{2}(\xi) = 0
\quad \forall v\in V,
\end{gather}
where $s(0)=1$ and $s(\ell_e)=-1$.
It is well known that $\mathcal{D}$ defined above is self-adjoint on $L^{2}(\mathcal{G},\mathbb{C}^{2})$ (for instance, see \cite{MR3934110,yang2025bound}).
\end{Def}

\subsection{Examples of periodic quantum graphs}

We give several basic examples of periodic quantum graphs fitting the assumptions above. They illustrate both one-dimensional periodic structures (with a free, cocompact action of $\mathbb{Z}$) and higher-dimensional analogues (with a free, cocompact action of $\mathbb{Z}^{d}$, $d\ge 2$).

\paragraph{Example 1 (Periodic chain graph).}
Let $\mathcal{G}$ be the infinite chain whose vertices are indexed by $\mathbb{Z}$ and each edge connects two consecutive integers. Each edge is identified with an interval of length $1$. We define the graph automorphism $T:\mathcal{G}\to\mathcal{G}$ by shifting the chain one edge to the right, that is, $T$ maps the vertex $k$ to $k+1$ and each edge $[k,k+1]$ to $[k+1,k+2]$. The action of $\mathbb{Z}$ generated by $T$ is free and by edgewise isometries, and the quotient $\mathcal{G}/\mathbb{Z}$ is a finite graph consisting of one vertex and one loop edge (the image of any edge $[k,k+1]$ under the identification of all vertices modulo $\mathbb{Z}$).
A fundamental cell is given by the edge
\[
\mathcal{K} = [0,1],
\]
so that
\[
\mathcal{G} = \bigcup_{k\in\mathbb{Z}} T^{k}(\mathcal{K}), \qquad
T^{k}(\mathcal{K})\cap T^{\ell}(\mathcal{K}) \text{ consists only of endpoints if } k\neq\ell.
\]

\begin{figure}[ht]
	\centering
	\begin{tikzpicture}[scale=1]
		\foreach \x in {-3,-2,-1,0,1,2,3}{
			\fill (\x,0) circle (2pt);
		}
		\foreach \x in {-3,-2,-1,0,1,2}{
			\draw (\x,0) -- (\x+1,0);
		}
		\draw[very thick] (0,0.15) -- (1,0.15);
		\node[above] at (0.5,0.15) {$\mathcal{K}$};
		\draw[->] (0.5,-0.4) -- (1.5,-0.4);
		\node[below] at (1,-0.4) {$T$};
		\node[below] at (-3,0) {$\cdots$};
		\node[below] at (3,0) {$\cdots$};
	\end{tikzpicture}
	\caption{A periodic chain graph with fundamental cell $\mathcal{K}=[0,1]$.}
	\label{fig:periodic-chain}
\end{figure}
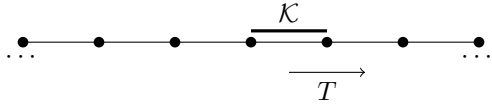

\paragraph{Example 2 (Periodic decorated chain).}
We modify the previous example by attaching to each vertex $k\in\mathbb{Z}$ an additional finite edge (stub) of fixed length $L>0$. The resulting graph is sometimes called a decorated chain. The automorphism $T$ is again the unit shift along the chain: it maps vertex $k$ to $k+1$ and carries each stub at $k$ to the stub at $k+1$. The action of $\mathbb{Z}$ is free and by isometries, and the quotient $\mathcal{G}/\mathbb{Z}$ is a finite graph consisting of one vertex, one loop edge (the horizontal period), and one additional edge (the stub) attached to that vertex. A convenient choice of fundamental cell is the union of the edge $[0,1]$ together with the stub attached at $0$.

\begin{figure}[ht]
	\centering
	\begin{tikzpicture}[scale=1]
		\foreach \x in {-2,-1,0,1,2}{
			\fill (\x,0) circle (2pt);
		}
		\foreach \x in {-2,-1,0,1}{
			\draw (\x,0) -- (\x+1,0);
		}
		\foreach \x in {-2,-1,0,1,2}{
			\draw (\x,0) -- (\x,0.8);
			\fill (\x,0.8) circle (1.5pt);
		}
		\draw[very thick] (0,0.15) -- (1,0.15);
		\draw[very thick] (0,0.05) -- (0,0.8);
		\node[above left] at (0,0.8) {$\mathcal{K}$};
		\draw[->] (0.4,-0.4) -- (1.4,-0.4);
		\node[below] at (0.9,-0.4) {$T$};
		\node[below] at (-2,0) {$\cdots$};
		\node[below] at (2,0) {$\cdots$};
	\end{tikzpicture}
	\caption{A periodic decorated chain graph. Each cell contains one horizontal edge and one stub.}
	\label{fig:decorated-chain}
\end{figure}
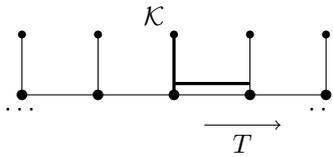

\paragraph{Example 3 (Periodic ladder graph).}
Consider the infinite ladder obtained by taking two parallel copies of the chain in Figure~\ref{fig:periodic-chain} and connecting corresponding vertices by vertical rungs. The automorphism $T$ is again the unit horizontal shift, acting simultaneously on both chains and on the rungs. The quotient $\mathcal{G}/\mathbb{Z}$ is a finite ladder cell with two vertices and three edges (two horizontal and one vertical). A fundamental cell $\mathcal{K}$ can be chosen as the connected subgraph consisting of the two horizontal edges between $x=0$ and $x=1$ together with the vertical rung at $x=0$.

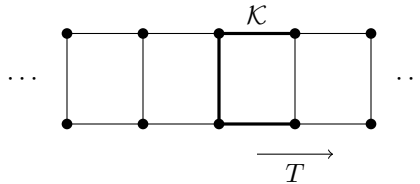
\begin{figure}[ht]
	\centering
	\begin{tikzpicture}[scale=1]
		\foreach \x in {-2,-1,0,1,2}{
			\fill (\x,0.6) circle (2pt);
			\fill (\x,-0.6) circle (2pt);
		}
		\foreach \x in {-2,-1,0,1}{
			\draw (\x,0.6) -- (\x+1,0.6);
			\draw (\x,-0.6) -- (\x+1,-0.6);
		}
		\foreach \x in {-2,-1,0,1,2}{
			\draw (\x,0.6) -- (\x,-0.6);
		}
		\draw[very thick] (0,0.6) -- (1,0.6);
		\draw[very thick] (0,-0.6) -- (1,-0.6);
		\draw[very thick] (0,0.6) -- (0,-0.6);
		\node[above] at (0.5,0.6) {$\mathcal{K}$};
		\draw[->] (0.5,-1.0) -- (1.5,-1.0);
		\node[below] at (1.0,-1.0) {$T$};
		\node[left] at (-2.2,0) {$\cdots$};
		\node[right] at (2.2,0) {$\cdots$};
	\end{tikzpicture}
	\caption{A periodic ladder graph with a natural $\mathbb{Z}$-action along the horizontal direction.}
	\label{fig:ladder}
\end{figure}

\paragraph{Example 4 (Multi-channel periodic strip).}
Consider three parallel chains (upper, middle, lower), each isomorphic to the chain in Example 1, and connect corresponding vertices by vertical edges. The automorphism $T$ is again the unit shift along the horizontal direction, acting on all three chains and on the vertical connections. The quotient $\mathcal{G}/\mathbb{Z}$ is a finite graph describing one transversal slice. A natural choice of fundamental cell $\mathcal{K}$ is the connected subgraph consisting of the three horizontal edges between $x=0$ and $x=1$ together with the two vertical connections at $x=0$.

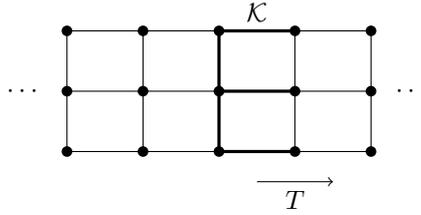
\begin{figure}[ht]
	\centering
	\begin{tikzpicture}[scale=1]
		\foreach \x in {-2,-1,0,1,2}{
			\fill (\x,0.8) circle (2pt);
			\fill (\x,0) circle (2pt);
			\fill (\x,-0.8) circle (2pt);
		}
		\foreach \x in {-2,-1,0,1}{
			\draw (\x,0.8) -- (\x+1,0.8);
			\draw (\x,0) -- (\x+1,0);
			\draw (\x,-0.8) -- (\x+1,-0.8);
		}
		\foreach \x in {-2,-1,0,1,2}{
			\draw (\x,0.8) -- (\x,0);
			\draw (\x,0) -- (\x,-0.8);
		}
		\draw[very thick] (0,0.8) -- (1,0.8);
		\draw[very thick] (0,0) -- (1,0);
		\draw[very thick] (0,-0.8) -- (1,-0.8);
		\draw[very thick] (0,0.8) -- (0,0);
		\draw[very thick] (0,0) -- (0,-0.8);
		\node[above] at (0.5,0.8) {$\mathcal{K}$};
		\draw[->] (0.5,-1.2) -- (1.5,-1.2);
		\node[below] at (1.0,-1.2) {$T$};
		\node[left] at (-2.2,0) {$\cdots$};
		\node[right] at (2.2,0) {$\cdots$};
	\end{tikzpicture}
	\caption{A multi-channel periodic strip graph with a single $\mathbb{Z}$-periodic direction.}
	\label{fig:strip}
\end{figure}

\paragraph{Example 5 (Two-dimensional periodic square lattice).}
Consider the square lattice embedded in $\mathbb{R}^{2}$ with vertex set
\[
\{(m,n): m,n\in\mathbb{Z}\}
\]
and edges connecting all nearest neighbours horizontally and vertically, each edge having length $1$. The group $\mathbb{Z}^{2}$ acts on $\mathcal{G}$ by integer translations
\[
(k_{1},k_{2})\cdot(m,n) = (m+k_{1}, n+k_{2}),
\]
and this action extends to a free, cocompact action by graph automorphisms.

A fundamental cell $\mathcal K$ can be chosen so that different translates intersect only at boundary vertices. For instance, let $\mathcal K$ be the connected subgraph consisting of the vertex $(0,0)$ together with the two edges joining $(0,0)$ to $(1,0)$ and $(0,0)$ to $(0,1)$ (including their endpoints). Then
\[
\mathcal{G} = \bigcup_{(k_1,k_2)\in\mathbb{Z}^2} (k_1,k_2)\cdot \mathcal K,
\]
and any two distinct translates intersect only at vertices. In this case the quotient graph $\mathcal{G}/\mathbb{Z}^{2}$ is a finite graph with one vertex and two loops.

\begin{figure}[ht]
	\centering
	\begin{tikzpicture}[scale=1]
		\foreach \x in {0,1,2}{
			\foreach \y in {0,1,2}{
				\fill (\x,\y) circle (2pt);
			}
		}
		\foreach \y in {0,1,2}{
			\foreach \x in {0,1}{
				\draw (\x,\y) -- (\x+1,\y);
			}
		}
		\foreach \x in {0,1,2}{
			\foreach \y in {0,1}{
				\draw (\x,\y) -- (\x,\y+1);
			}
		}
		\draw[very thick] (0,0) -- (1,0);
		\draw[very thick] (0,0) -- (0,1);
		\node[below left] at (0,0) {$\mathcal{K}$};
		\draw[->] (0.5,-0.4) -- (1.5,-0.4);
		\node[below] at (1.0,-0.4) {$T_{1}$};
		\draw[->] (-0.4,0.5) -- (-0.4,1.5);
		\node[left] at (-0.4,1.0) {$T_{2}$};
	\end{tikzpicture}
	\caption{A two-dimensional periodic square lattice graph with a free $\mathbb{Z}^{2}$-action.}
	\label{fig:square-lattice}
\end{figure}
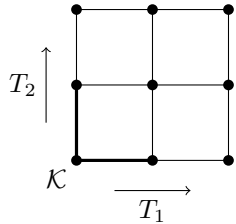

These examples show that our abstract setting naturally covers graphs endowed with a free, cocompact action of a discrete translation group. In the rest of the paper we will mainly work with quantum graphs $\mathcal{G}$ carrying a free, cocompact action of $\mathbb{Z}^{d}$, $d\ge 1$, by graph automorphisms, and we will refer to such graphs as $\mathbb{Z}^{d}$-periodic quantum graphs. The chain, decorated chain, ladder and strip above correspond to the case $d=1$, while the square lattice in Figure~\ref{fig:square-lattice} is a basic example for $d=2$.
	
	\subsection{The associated quadratic form}
	
	Following the approach in \cite{MR3934110}, we define the space
\begin{equation}\label{eq-2.5}
	Y= \big[ L^{2}\big(\mathcal{G},\mathbb{C}^{2}\big),\,\operatorname{dom}(\mathcal{D}) \big]_{1/2},
\end{equation}
namely, the real interpolation space of order $\frac{1}{2}$ between $L^{2}$ and the domain of the Dirac operator $\mathcal{D}$, where $\operatorname{dom}(\mathcal D)$ is endowed with the graph norm.

Since $\operatorname{dom}(\mathcal{D})$ is continuously embedded into the edgewise Sobolev space
\[
H^{1}\big(\mathcal{G},\mathbb{C}^{2}\big)
= \bigoplus_{e\in E} H^{1}\big(I_{e},\mathbb{C}^{2}\big),
\]
interpolation theory yields
\[
Y \hookrightarrow  H^{1/2}\big(\mathcal{G},\mathbb{C}^{2}\big)
:= \bigoplus_{e\in E} H^{1/2}\big(I_{e},\mathbb{C}^{2}\big)
= \big[ L^{2}\big(\mathcal{G},\mathbb{C}^{2}\big),\, H^{1}\big(\mathcal{G},\mathbb{C}^{2}\big)\big]_{1/2},
\]
where $Y$ is endowed with the interpolation norm defined in \eqref{eq-2.7} below and
$ H^{1/2}(\mathcal{G},\mathbb{C}^{2})$ with its natural norm.
Using one-dimensional fractional Sobolev embeddings on each edge and the periodic structure of $\mathcal{G}$, we obtain
\[
Y \hookrightarrow L^{p}\big(\mathcal{G},\mathbb{C}^{2}\big)
\quad\text{for every }2\le p<\infty.
\]
Moreover, for every fixed compact subgraph $\mathcal K\subset\mathcal G$, the embedding
\[
Y \hookrightarrow L^{p}\big(\mathcal{K},\mathbb{C}^{2}\big)
\]
is compact for every $2\le p<\infty$.

Next, we claim that
\begin{equation}\label{eq-2.6}
	\operatorname{dom}\big(\mathcal{Q}_{\mathcal{D}}\big) = Y,
\end{equation}
where $\mathcal{Q}_{\mathcal{D}}$ denotes the (closed) quadratic form associated with the self-adjoint operator $\mathcal{D}$, defined on the form domain $D(|\mathcal D|^{1/2})$.

To prove \eqref{eq-2.6}, we use the spectral theorem in \cite{MR493419} in the following form.

	\begin{Thm}\label{theo-2.1}
		Let $H$ be a self-adjoint operator on a separable Hilbert space $\mathcal{H}$
		with domain $\operatorname{dom}(H)$.
		Then there exist a \(\sigma\)-finite measure space \((M,\mu)\), a unitary operator
		\(U:\mathcal H\to L^{2}(M,\mu)\), and a real-valued measurable function $f$ on $M$, finite almost everywhere, such that:
		\begin{itemize}
			\item[(1)] $\psi\in\operatorname{dom}(H)$ if and only if $f\,U\psi\in L^{2}(M,\mu)$;
			\item[(2)] if $\varphi\in U(\operatorname{dom}(H))$, then
			\[
			(UHU^{-1}\varphi)(m) = f(m)\,\varphi(m)
			\quad\text{for almost every }m\in M.
			\]
		\end{itemize}
	\end{Thm}
	
	The above theorem states that $H$ is unitarily equivalent to the multiplication operator
	by $f$ (which we still denote by $f$) on the space $L^{2}(M,\mu)$, whose domain is
	\[
	\operatorname{dom}(f)
	= \big\{\varphi\in L^{2}(M,\mu): f(\cdot)\,\varphi(\cdot)\in L^{2}(M,\mu)\big\},
	\]
	endowed with the graph norm
	\[
	\|\varphi\|_{1}^{2}
	= \int_{M} \big(1+f(m)^{2}\big)\,|\varphi(m)|^{2}\,d\mu(m).
	\]
	Since $f$ is a multiplication operator, the quadratic form associated with $H$ (or equivalently with $|H|$)
	has form domain
	\[
	\operatorname{dom}(\mathfrak{q}_{H})
	= \big\{\varphi\in L^{2}(M,\mu): \sqrt{|f(\cdot)|}\,\varphi(\cdot)\in L^{2}(M,\mu)\big\}.
	\]
	We will prove below that this coincides with the interpolation space of order $\frac{1}{2}$ between
	$L^{2}(M,\mu)$ and $\operatorname{dom}(f)$, which will imply \eqref{eq-2.6}.
	
	Consider the Hilbert spaces $\mathcal{H}_{0}=L^{2}(M,\mu)$, with norm
	$\|x\|_{0}=|x|_{L^{2}(M,\mu)}$, and $\mathcal{H}_{1}=\operatorname{dom}(f)$, with norm $\|\cdot\|_{1}$ as above. Then $\mathcal{H}_{1}\subset\mathcal{H}_{0}$, and the squared norm
	$\|x\|_{1}^{2}$ defines a densely defined quadratic form on $\mathcal{H}_{0}$, represented by
	\[
	\|x\|_{1}^{2} = \big\langle (1+f(\cdot)^{2})x,x\big\rangle_{0},
	\]
	where $\langle\cdot,\cdot\rangle_{0}$ is the scalar product of $\mathcal{H}_{0}$.
	
	We now introduce the quadratic version of Peetre's $K$-functional:
	\[
	K(t,x)
	= \inf\Big\{\|x_{0}\|_{0}^{2} + t\|x_{1}\|_{1}^{2}:
	x = x_{0}+x_{1},\ x_{0}\in\mathcal{H}_{0},\ x_{1}\in\mathcal{H}_{1}\Big\},
	\]
	for $t>0$ and $x\in\mathcal{H}_{0}$.
	According to standard arguments (see, e.g., \cite{MR4019467} and the references therein),
	the intermediate spaces
	\[
	\mathcal{H}_{\theta}= [\mathcal{H}_{0},\mathcal{H}_{1}]_{\theta} \subset \mathcal{H}_{0},
	\qquad 0<\theta<1,
	\]
	are given by those $x\in\mathcal{H}_{0}$ such that the following quantity is finite:
	\begin{equation}\label{eq-2.7}
		\|x\|_{\theta}^{2}
		= \int_{0}^{\infty} \big(t^{-\theta} K(t,x)\big)\,\frac{dt}{t} < \infty.
	\end{equation}
	In particular, for $\theta=\frac{1}{2}$ we obtain the interpolation space
	$\mathcal{H}_{1/2} = [\mathcal{H}_{0},\mathcal{H}_{1}]_{1/2}$ endowed with the norm $\|\cdot\|_{1/2}$ defined by \eqref{eq-2.7}.
	
	Next, we prove the following representation formula.
	
	\begin{Pro}\label{pro-2.1}
		For every $\theta\in(0,1)$ there exists a constant $C_{\theta}>0$ such that
		\[
		\|x\|_{\theta}^{2}
		= C_{\theta}\,\big\langle \big(1+f(\cdot)^{2}\big)^{\theta}x,x\big\rangle_{0},
		\qquad x\in\mathcal{H}_{\theta},
		\]
		where $\mathcal{H}_{\theta}=[\mathcal{H}_{0},\mathcal{H}_{1}]_{\theta}$.
	\end{Pro}
	
	\begin{proof}
		For simplicity set
		\[
		D= 1+f(\cdot)^{2}.
		\]
		Then $D$ is a positive self-adjoint operator on $\mathcal{H}_{0}$ with dense domain, and its
		positive square root $D^{1/2}$ has a dense domain $\mathbb{D}:=\operatorname{dom}(D^{1/2})$ in $\mathcal{H}_{0}$.
		
		We first identify the $K$-functional. For $x\in\mathbb{D}$ and $t>0$, there exists a unique
		decomposition
		\[
		x = x_{0,t} + x_{1,t},
		\]
		with $x_{0,t},x_{1,t}\in\mathbb{D}$, such that
		\begin{equation}\label{eq-2.9}
			K(t,x) = \|x_{0,t}\|_{0}^{2} + t\|x_{1,t}\|_{1}^{2}.
		\end{equation}
		For any $y\in\mathbb{D}$ and $s\in\mathbb{R}$, consider the perturbed decomposition
		\[
		x = (x_{0,t}+sy) + (x_{1,t}-sy).
		\]
		By the minimality of $(x_{0,t},x_{1,t})$ in the definition of $K(t,x)$, the function
		\[
		s\mapsto \|x_{0,t}+sy\|_{0}^{2} + t\|x_{1,t}-sy\|_{1}^{2}
		\]
		has vanishing derivative at $s=0$, that is,
		\[
		\frac{d}{ds}\Big(\|x_{0,t}+sy\|_{0}^{2}
		+ t\|x_{1,t}-sy\|_{1}^{2}\Big)\Big|_{s=0} = 0.
		\]
		Using $\|x\|_{1}^{2}=\|D^{1/2}x\|_{0}^{2}$ and the self-adjointness of $D$, this yields
		\[
		\big\langle x_{0,t}-tDx_{1,t},y\big\rangle_{0} = 0
		\quad\forall\,y\in\mathbb{D},
		\]
		hence
		\[
		x_{0,t} = tDx_{1,t}.
		\]
		Since $x = x_{0,t}+x_{1,t}$, we obtain
		\begin{equation}\label{eq-2.10}
			x_{1,t} = (1+tD)^{-1}x,
			\qquad
			x_{0,t} = tD(1+tD)^{-1}x.
		\end{equation}
		Substituting \eqref{eq-2.10} into \eqref{eq-2.9} and using the functional calculus for $D$, we
		get, first for $x\in\mathbb{D}$ and then by density for all $x\in\mathcal{H}_{0}$,
		\begin{equation}\label{eq-2.8}
			K(t,x) = \big\langle \frac{tD}{1+tD}x,x\big\rangle_{0}
			\qquad\forall\,t>0.
		\end{equation}
		
		We now compute the interpolation norm. By \eqref{eq-2.7} and \eqref{eq-2.8},
		\begin{equation}\label{eq-2.11}
			\begin{aligned}
				\|x\|_{\theta}^{2}
				&= \int_{0}^{\infty} t^{-\theta} K(t,x)\,\frac{dt}{t} \\
				&= \int_{0}^{\infty} t^{-\theta}\big\langle \tfrac{tD}{1+tD}x,x\big\rangle_{0}\,\frac{dt}{t} \\
				&= \int_{0}^{\infty} t^{-\theta}\big\langle \tfrac{D}{1+tD}x,x\big\rangle_{0}\,dt \\
				&= \big\langle D\Big(\int_{0}^{\infty} \frac{dt}{t^{\theta}(1+tD)}\Big)x,x\big\rangle_{0}.
			\end{aligned}
		\end{equation}
		For $d>0$, set
		\[
		g(d)= \int_{0}^{\infty} \frac{dt}{t^{\theta}(1+td)}.
		\]
		By the change of variables $s=td$ we obtain
		\[
		g(d) = d^{\theta-1} \int_{0}^{\infty} \frac{ds}{s^{\theta}(1+s)}
		= C_{\theta}\,d^{\theta-1},
		\]
		where
		\[
		C_{\theta}= \int_{0}^{\infty} \frac{ds}{s^{\theta}(1+s)} \in (0,+\infty)
		\]
		depends only on $\theta$. By the functional calculus,
		\[
		\int_{0}^{\infty} \frac{dt}{t^{\theta}(1+tD)}
		= g(D) = C_{\theta}\,D^{\theta-1},
		\]
		and \eqref{eq-2.11} becomes
		\[
		\|x\|_{\theta}^{2}
		= \big\langle D\,g(D)x,x\big\rangle_{0}
		= C_{\theta}\big\langle D^{\theta}x,x\big\rangle_{0}.
		\]
		This proves the claim.
	\end{proof}
	
	It is worth noting that, according to the previous proposition, if we take \(\theta=\frac12\),
	then \(\mathcal H_{1/2}\) coincides with the form domain of the operator \(H\) (equivalently of \(|H|\)).
	Indeed,
	\[
	x\in\mathcal H_{1/2}
	\quad\Longleftrightarrow\quad
	\langle D^{1/2}x,x\rangle_0
	= \int_M \sqrt{1+f(m)^{2}}\,|x(m)|^{2}\,d\mu(m) < \infty.
	\]
	Since
	\[
	|f(m)| \le \sqrt{1+f(m)^{2}} \le 1+|f(m)| \quad\text{for all }m\in M,
	\]
	the above condition is equivalent to
	\[
	\int_M |f(m)|\,|x(m)|^{2}\,d\mu(m) < \infty,
	\]
	which is precisely the condition \(x\in D(|H|^{1/2})\), i.e. the form domain of \(H\).
	Therefore, setting \(H=\mathcal D\) and \(\mathcal H=L^{2}(\mathcal G,\mathbb C^{2})\), we see that
	\eqref{eq-2.5} is exactly the form domain of \(\mathcal D\), and hence \eqref{eq-2.6} holds.

Finally, in view of \eqref{eq-2.6} we denote the form domain of $\mathcal D$ by $Y$.
Since we work with real Fr\'echet derivatives,
we regard $L^{2}(\mathcal G,\mathbb C^{2})$ and $Y$ as real Hilbert spaces endowed with the real inner product
\[
(u,v)_{2}
= \operatorname{Re}\int_{\mathcal G}\langle u(x),v(x)\rangle_{\mathbb C^{2}}\,dx.
\]
Let $\operatorname{sgn}(\mathcal D)$ be the bounded self-adjoint operator given by the functional calculus,
\[
\operatorname{sgn}(\mathcal D)=\mathcal D|\mathcal D|^{-1}\quad \text{on }(\ker\mathcal D)^{\perp},
\qquad
\operatorname{sgn}(\mathcal D)=0\ \text{on }\ker\mathcal D.
\]
We define the (closed) symmetric bilinear form on $Y$ by
\[
\mathcal{Q}_{\mathcal{D}}(u,v)
= \big(|\mathcal D|^{1/2}u,\ \operatorname{sgn}(\mathcal D)\,|\mathcal D|^{1/2}v\big)_{2},
\qquad
\mathcal{Q}_{\mathcal{D}}(u)=\mathcal{Q}_{\mathcal{D}}(u,u).
\]
If $u,v\in \operatorname{dom}(\mathcal D)$, then
\[
\mathcal{Q}_{\mathcal{D}}(u,v)
= \operatorname{Re}\int_{\mathcal G}\langle u,\mathcal D v\rangle_{\mathbb C^{2}}\,dx,
\]
so the above definition extends the usual integral identity to the whole form domain.
	
	\subsection{Critical point theorems}
In this section we recall some abstract critical point theory following \cite{MR2255874}.
Let $Z$ be a real Banach space with a topological direct sum decomposition $Z = M \oplus N$,
where $M$ and $N$ are closed subspaces and the corresponding projections $P_{M},P_{N}$ onto $M,N$ are bounded.

For a functional $\Phi\in C^{1}(Z,\mathbb{R})$ we write
\[
\Phi_{a} = \{u\in Z : \Phi(u)\ge a\},\quad
\Phi^{b}= \{u\in Z : \Phi(u)\le b\},\quad
\Phi_{a}^{b} = \Phi_{a}\cap\Phi^{b}.
\]

\begin{Def}\label{def-2.3}
A sequence $(u_{n})\subset Z$ is called a $(C)_{c}$-sequence if
$\Phi(u_{n})\to c$ and $(1+\|u_{n}\|)\,\|\Phi'(u_{n})\|_{Z^{*}}\to 0$.
Moreover, $\Phi$ is said to satisfy the $(C)_{c}$-condition if any
$(C)_{c}$-sequence has a convergent subsequence.
\end{Def}

\begin{Def}\label{def-2.4}
A set $\mathcal{A}\subset Z$ is said to be a $(C)_{c}$-attractor if for any
$\varepsilon,\delta>0$ and any $(C)_{c}$-sequence $(u_{n})$ there exists
$n_{0}$ such that
\[
u_{n} \in U_{\varepsilon}\big(\mathcal{A}\cap\Phi_{c-\delta}^{c+\delta}\big)
\quad\text{for all }n\ge n_{0},
\]
where $U_{\varepsilon}(\cdot)$ denotes the $\varepsilon$-neighbourhood in $Z$.
Given an interval $I\subset\mathbb{R}$, $\mathcal{A}$ is said to be a
$(C)_{I}$-attractor if it is a $(C)_{c}$-attractor for every $c\in I$.
\end{Def}

From now on we assume that $M$ is separable and reflexive, and we fix a countable dense subset
$\mathcal{S}\subset M^{*}$. For each $s\in\mathcal{S}$ we define a seminorm on $Z$ by
\[
p_{s}:Z\to\mathbb{R},\qquad
p_{s}(u) = |s(P_M u)| + \|P_N u\|.
\]
We denote by $\mathcal{T}_{\mathcal{S}}$ the topology on $Z$ induced by the family
$\{p_{s}\}_{s\in\mathcal{S}}$. We also denote by $\mathcal{T}_{w}$ the weak topology on $Z$
and by $\mathcal{T}_{w^{*}}$ the weak\(*\)-topology on $Z^{*}$.

	We shall use the following assumptions on $\Phi$:
	
	\begin{itemize}
		\item[(\(\Phi_{0}\))] For any $c\in\mathbb{R}$, the set $\Phi_{c}$ is
		$\mathcal{T}_{\mathcal{S}}$-closed and
		\[
		\Phi'\colon (\Phi_{c},\mathcal{T}_{\mathcal{S}})\to (Z^{*},\mathcal{T}_{w^{*}})
		\]
		is continuous.
		
		\item[(\(\Phi_{1}\))] For any $c>0$ there exists $\zeta>0$ such that
		\[
		\|u\| < \zeta\,\|P_{N}u\|\quad\text{for all }u\in\Phi_{c}.
		\]
		
		\item[(\(\Phi_{2}\))] There exists $\rho>0$ such that
		\[
		\eta= \inf\Phi(S_{\rho}^{N}) > 0,\qquad
		S_{\rho}^{N} := \{u\in N : \|u\|=\rho\}.
		\]
		
		\item[(\(\Phi_{3}\))] There exists a finite-dimensional subspace $N_{0}\subset N$
		and $R>\rho$ such that, setting $E_{0}=M\oplus N_{0}$ and
		\[
		B_{0}= \{u\in E_{0} : \|u\|\le R\},
		\]
		we have $b'= \sup\Phi(E_{0})<\infty$ and
		\[
		\sup \Phi(E_{0}\setminus B_{0}) < \inf \Phi(B_{\rho}\cap N),
		\]
		where $B_{\rho}= \{u\in Z : \|u\|\le \rho\}$.
		
		\item[(\(\Phi_{4}\))] There exist an increasing sequence \(N_{n}\subset N\) of
		finite-dimensional subspaces and a sequence \((R_{n})\) of positive numbers such
		that, setting \(E_{n} = M\oplus N_{n}\) and \(B_{n} = B_{R_{n}}\cap E_{n}\),
		\[
		\sup\Phi(E_{n}) < \infty
		\quad\text{and}\quad
		\sup\Phi(E_{n}\setminus B_{n}) < \inf\Phi(B_{\rho}\cap N).
		\]
	
		\item[(\(\Phi_{5}\))] One of the following holds:
		\begin{itemize}
			\item[(i)] for any interval $I\subset (0,\infty)$ there exists a $(C)_{I}$-attractor
			$\mathcal{A}$ such that $P_{N}\mathcal{A}$ is bounded and
			\[
			\inf\big\{\|P_{N}(u-v)\| : u,v\in\mathcal{A},\,P_{N}(u-v)\neq 0\big\} > 0;
			\]
			\item[(ii)] $\Phi$ satisfies the $(C)_{c}$-condition for all $c>0$.
		\end{itemize}
	\end{itemize}

	\begin{Thm}\label{theo-2.2}
Assume $(\Phi_{0})$--$(\Phi_{2})$ and suppose that there exist $R>\rho>0$ and $e\in N$
with $\|e\|=1$ such that
\[
\sup\Phi(\partial Q) < \eta,
\]
where $\eta=\inf\Phi(S_{\rho}^{N})>0$ is given by $(\Phi_2)$ and
\[
Q := \{u = x+te : x\in M,\ t\ge 0,\ \|u\|\le R\}.
\]
Then $\Phi$ has a $(C)_{c}$-sequence with
\[
\eta \le c \le \sup\Phi(Q).
\]
\end{Thm}

	\begin{Thm}\label{theo-2.3}
		Assume that $\Phi$ is even with $\Phi(0)=0$ and that $(\Phi_{0})$--$(\Phi_{5})$ are
		satisfied. Then $\Phi$ possesses an unbounded sequence of positive critical values.
	\end{Thm}
	
	\begin{Thm}\label{theo-2.4}
		Suppose $\Phi\in C^{1}(Z,\mathbb{R})$ is of the form
		\[
		\Phi(u) = \frac{1}{2}\big(\|y\|^{2} - \|x\|^{2}\big) - \Psi(u)
		\quad\text{for }u=x+y\in Z=M\oplus N,
		\]
		such that:
		\begin{itemize}
			\item[(i)] $\Psi\in C^{1}(Z,\mathbb{R})$ is bounded from below;
			\item[(ii)] $\Psi\colon (Z,\mathcal{T}_{w})\to\mathbb{R}$ is sequentially lower
			semicontinuous, that is, $u_{n}\rightharpoonup u$ in $Z$ implies
			\(\Psi(u)\le \liminf\Psi(u_{n})\);
			\item[(iii)] $\Psi'\colon (Z,\mathcal{T}_{w})\to (Z^{*},\mathcal{T}_{w^{*}})$ is
			sequentially continuous;
			\item[(iv)] $v\colon Z\to\mathbb{R}$, $v(u)=\|u\|^{2}$, is $C^{1}$ and
			$v'\colon (Z,\mathcal{T}_{w})\to (Z^{*},\mathcal{T}_{w^{*}})$ is
			sequentially continuous.
		\end{itemize}
		Then $\Phi$ satisfies $(\Phi_{0})$.
	\end{Thm}
	
	Now, we recall the following Brezis--Lieb type lemma and Concentration--Compactness Principle on periodic quantum graphs \cite{yang2025bound}.
	\begin{Lem}\label{lemma2.1}
		Let $(\mathcal{G},\mu)$ be the quantum graph endowed with the one-dimensional
		Lebesgue measure along edges. Let $\{u_n\}\subset L^p(\mathcal{G})$ for some $1<p<\infty$
		and assume
		\begin{enumerate}
			\item[\textnormal{(i)}] $u_n(x)\to u(x)$ almost everywhere on $\mathcal{G}$;
			\item[\textnormal{(ii)}] $\sup_n|u_n|_{L^p(\mathcal{G})}<\infty$.
		\end{enumerate}
		Then, as $n\to\infty$,
		\[
		|u_n|_{L^p(\mathcal{G})}^p
		= |u_n-u|_{L^p(\mathcal{G})}^p + |u|_{L^p(\mathcal{G})}^p + o(1).
		\]
	\end{Lem}
	
\begin{Lem}\label{lem:2.2}
Let $\mathcal G$ be a periodic quantum graph with a free cocompact $\mathbb Z^{d}$-action,
and let $\mathcal K\subset\mathcal G$ be a fixed fundamental cell.
Assume that the induced action $u\mapsto a*u$ is isometric on $Y$.
Let $(u_m)$ be bounded in $Y$.
Then exactly one of the following alternatives holds:
\begin{itemize}
\item[(i)] Vanishing: For every $R>0$,
\[
\sup_{x\in\mathcal G}\int_{B_R(x)}|u_m|^{2}\,dx\to 0.
\]
In this case $u_m\to 0$ in $L^{q}(\mathcal G,\mathbb C^{2})$ for every $q\in(2,\infty)$.

\item[(ii)] Compactness modulo $\mathbb Z^{d}$-translations: There exist $R>0$, $\delta>0$,
a sequence $(a_m)\subset\mathbb Z^{d}$ and a nonzero $u\in Y$ such that, setting $v_m=a_m*u_m$,
\[
\int_{B_R(\bar x)}|v_m|^{2}\,dx\ge \delta \quad\text{for all }m
\]
for some fixed $\bar x\in\mathcal K$, and
\[
v_m\rightharpoonup u\ \text{in }Y,
\qquad
v_m\to u\ \text{in }L^{q}_{\mathrm{loc}}(\mathcal G,\mathbb C^{2})\ \text{for all }q\in[2,\infty).
\]
\end{itemize}
\end{Lem}
\begin{proof}
\textbf{Step 1: If $(i)$ fails, then $(ii)$ holds.} 
Assume that $(i)$ fails. Then there exist $R>0$, $\delta>0$ and points $x_m\in\mathcal G$ such that
\[
\int_{B_R(x_m)}|u_m|^{2}\,dx \ge \delta \quad\text{for all }m.
\]
By cocompactness of the free $\mathbb Z^{d}$-action, we can choose $a_m\in\mathbb Z^{d}$ such that
\[
y_m = T^{-a_m}x_m \in \mathcal K.
\]
Set $v_m = a_m*u_m$. By the isometry of the action on $Y$, $(v_m)$ is bounded in $Y$ and
\[
\int_{B_R(y_m)}|v_m|^{2}\,dx
=\int_{B_R(x_m)}|u_m|^{2}\,dx \ge \delta.
\]
Since $\mathcal K$ is compact, up to a subsequence $y_m\to\bar x\in\mathcal K$.
For all large $m$ we have $B_R(y_m)\subset B_{2R}(\bar x)$, hence
\[
\int_{B_{2R}(\bar x)}|v_m|^{2}\,dx \ge \delta \quad\text{for all large }m.
\]
Up to a subsequence, $v_m\rightharpoonup u$ in $Y$.

Let $G_0=B_{2R}(\bar x)$, which is a compact subgraph (a finite union of edge segments).
By local compactness of the embedding on compact subgraphs, we have
\[
v_m\to u \quad\text{in }L^{2}(G_0,\mathbb C^{2}),
\]
hence
\[
\int_{B_{2R}(\bar x)}|u|^{2}\,dx
=\lim_{m\to\infty}\int_{B_{2R}(\bar x)}|v_m|^{2}\,dx \ge \delta,
\]
so $u\ne 0$.

Now let $G_1\subset\mathcal G$ be any compact subgraph. By boundedness of $(v_m)$ in $Y$ and the same
local compactness argument, we obtain
\[
v_m\to u \quad\text{in }L^{q}(G_1,\mathbb C^{2}) \ \text{for all }q\in[2,\infty).
\]
Therefore $v_m\to u$ in $L^{q}_{\mathrm{loc}}(\mathcal G,\mathbb C^{2})$ for all $q\in[2,\infty)$.
Renaming $2R$ as $R$, we get (ii).

\textbf{Step 2: If $(i)$ holds, then $u_m\to0$ in $L^{q}(\mathcal G,\mathbb C^{2})$ for every $q\in(2,\infty)$.}
Assume $(i)$. Fix $R>0$.
Choose finitely many points $x_1,\dots,x_J\in\mathcal K$ such that
\[
\mathcal K\subset \bigcup_{j=1}^{J} B_R(x_j).
\]
For each $a\in\mathbb Z^{d}$ set $\mathcal K_a = T^{a}\mathcal K$. Then
\[
\mathcal K_a \subset \bigcup_{j=1}^{J} B_R(T^{a}x_j),
\]
hence
\[
\sup_{a\in\mathbb Z^{d}}\int_{\mathcal K_a}|u_m|^{2}\,dx
\le J\sup_{x\in\mathcal G}\int_{B_R(x)}|u_m|^{2}\,dx \to 0.
\]
Define
\[
\mu_m=\sup_{a\in\mathbb Z^{d}}\int_{\mathcal K_a}|u_m|^{2}\,dx.
\]
Then $\mu_m\to0$.

We claim that for every $q\in(2,\infty)$ there exists $C_q>0$, independent of $m$, such that
\begin{equation}\label{eq:Lions-cell}
\|u_m\|_{L^{q}(\mathcal G)}^{q}
\le C_q\,\mu_m^{\frac{q-2}{2}}\,\|u_m\|_{Y}^{2}.
\end{equation}
Indeed, since all cells $\mathcal K_a$ are isometric to $\mathcal K$ and the action is isometric on $Y$,
the local Sobolev embedding constants on $\mathcal K_a$ are uniform in $a$.
Using the standard estimate on the compact cell $\mathcal K_a$,
\[
\|w\|_{L^{q}(\mathcal K_a)}^{q}
\le C_q \Bigl(\int_{\mathcal K_a}|w|^{2}\,dx\Bigr)^{\frac{q-2}{2}}
\|w\|_{Y(\mathcal K_a)}^{2},
\]
and summing over $a\in\mathbb Z^{d}$,
we obtain
\[
\|u_m\|_{L^{q}(\mathcal G)}^{q}
=\sum_{a\in\mathbb Z^{d}}\|u_m\|_{L^{q}(\mathcal K_a)}^{q}
\le C_q\,\mu_m^{\frac{q-2}{2}} \sum_{a\in\mathbb Z^{d}} \|u_m\|_{Y(\mathcal K_a)}^{2}.
\]
By locality of the $Y$-norm, the last sum is controlled by
$\|u_m\|_{Y}^{2}$ up to a constant independent of $m$, which gives \eqref{eq:Lions-cell}. Since $(u_m)$ is bounded in $Y$ and $\mu_m\to0$, \eqref{eq:Lions-cell} yields $\|u_m\|_{L^{q}(\mathcal G)}\to0$ for all $q\in(2,\infty)$.

\textbf{Step 3: Mutual exclusivity.}
If $(ii)$ holds, then for some $R>0$ and $\delta>0$ we have
$\int_{B_R(\bar x)}|v_m|^{2}\,dx\ge\delta$, hence (i) cannot hold. Therefore exactly one alternative occurs.
\end{proof}

	\section{Spectral properties and main results with potential $V_{1}$}
	
	\subsection{Spectral properties}
	Throughout this section we regard $L^{2}(\mathcal{G},\mathbb{C}^{2})$ as a complex Hilbert space
with the standard sesquilinear inner product
\[
(u,v)_{L^{2}}=\int_{\mathcal{G}}\langle u(x),v(x)\rangle_{\mathbb{C}^{2}}\,dx.
\]
When passing to the variational setting (real Fr\'echet derivatives via $\mathbb{C}^{2}\simeq\mathbb{R}^{4}$),
we will use the real part of this pairing.

	We denote by $|\cdot|_{p}$ the $L^{p}$-norm on $\mathcal{G}$ for $p\in[1,\infty]$. In the rest of this section we work under assumption~$(V_{1})$, in particular $V\in C^{1}(\mathcal G,[0,\infty))$ and is $\mathbb Z^{d}$-periodic.
We consider the Dirac-type operator
	\[
	A= -i\sigma_{1}\frac{d}{dx} + \big(V(x)+a\big)\sigma_{3}
	\]
	acting on $L^{2}(\mathcal{G},\mathbb{C}^{2})$.
	Under the Kirchhoff-type vertex conditions introduced in Definition~\ref{def2.2}, $A$ is self-adjoint on $L^{2}(\mathcal{G},\mathbb{C}^{2})$ (see, e.g., \cite{MR1064315}).
	It is unbounded both from above and from below.
    In order to analyze the spectrum of $A$, we record the edgewise differential expression of $A^{2}$:
for every edge $e\in E$ and $u\in \operatorname{dom}(A^{2})$,
\[
(A^{2}u)_{e}
= -u_{e}'' + \big(V+a\big)^{2}u_{e} + i\sigma_{3}\sigma_{1}\,V' u_{e}
\quad \text{on } I_{e},
\]
where $V'$ denotes the derivative of $V$ along the edge coordinate.
	Furthermore, for our purposes, we denote by $\sigma(S)$, $\sigma_{d}(S)$, $\sigma_{e}(S)$ and $\sigma_{c}(S)$
	the spectrum, the discrete spectrum, the essential spectrum and the continuous spectrum of a self-adjoint
	operator $S$ on $L^{2}$, respectively.
	
	\begin{Lem}\label{lem-3.1}
		If $(V_{1})$ holds, then $\sigma(A^{2})\subset[a^{2},\infty)$.
	\end{Lem}
	\begin{proof}
Assume $(V_{1})$. Since $A$ is self-adjoint, for every $u\in\operatorname{dom}(A)$ we have
\[
(A^{2}u,u)_{L^{2}}=(Au,Au)_{L^{2}}=\|Au\|_{L^{2}}^{2}.
\]
Write $A= B + (V+a)\sigma_{3}$ with $B=-i\sigma_{1}\frac{d}{dx}$. Then
\[
\|Au\|_{L^{2}}^{2}
= \|Bu+V\sigma_{3}u\|_{L^{2}}^{2}
+ a^{2}\|u\|_{L^{2}}^{2}
+ 2a\,(V u,u)_{L^{2}}
+ 2a\,\Re(Bu,\sigma_{3}u)_{L^{2}}.
\]
We claim that $\Re(Bu,\sigma_{3}u)_{L^{2}}=0$. Indeed, on each edge $e$,
an integration by parts gives
\[
(Bu_e,\sigma_{3}u_e)_{L^{2}(I_e)} + (\sigma_{3}u_e,Bu_e)_{L^{2}(I_e)}
= i\Big[u_e^{1}\overline{u_e^{2}}-u_e^{2}\overline{u_e^{1}}\Big]_{0}^{\ell_e}.
\]
Summing over all edges and using the Kirchhoff-type vertex conditions
\eqref{eq-2.3}--\eqref{eq-2.4}, all boundary contributions cancel, hence
\[
(Bu,\sigma_{3}u)_{L^{2}} + (\sigma_{3}u,Bu)_{L^{2}} = 0,
\]
which is equivalent to $\Re(Bu,\sigma_{3}u)_{L^{2}}=0$.

Therefore,
\[
\|Au\|_{L^{2}}^{2}
\ge a^{2}\|u\|_{L^{2}}^{2} + 2a\,(V u,u)_{L^{2}}
\ge a^{2}\|u\|_{L^{2}}^{2},
\]
since $(V_{1})$ implies $V\ge 0$. Hence $A^{2}\ge a^{2}I$ in the sense of quadratic forms,
and consequently $\sigma(A^{2})\subset[a^{2},\infty)$.
\end{proof}
	
	Denote by $\mathcal{D}(A)$ the domain of $A$. Equipped with the inner product
	\[
	(u,v)_{\mathcal{D}(A)} = (Au,Av)_{L^{2}} + (u,v)_{L^{2}},
	\]
	the space $\mathcal{D}(A)$ is a Hilbert space; we denote the corresponding norm by
	\[
	\|u\|_{\mathcal{D}(A)}^{2} = |Au|_{2}^{2} + |u|_{2}^{2}.
	\]
	
	\begin{Lem}\label{lem-3.2}
If $(V_{1})$ is satisfied, then on $\mathcal{D}(A)$ the graph norm $\|\cdot\|_{\mathcal{D}(A)}$
and the edgewise $H^{1}$-norm $\|\cdot\|_{ H^{1}(\mathcal{G},\mathbb{C}^{2})}$ are equivalent.
\end{Lem}

\begin{proof}
Recall that
\[
A u = -i\sigma_{1}u' + \big(V(x)+a\big)\sigma_{3}u
\quad\text{for }u\in\mathcal{D}(A).
\]
Assumption $(V_{1})$ gives \(V\in C^{1}(\mathcal{G},[0,\infty))\) and \(V(T^{k}x)=V(x)\) for all
\(k\in\mathbb{Z}^{d}\). Since the quotient \(\mathcal{G}/\mathbb{Z}^{d}\) is compact and \(V\) is continuous,
\(V\) is bounded on \(\mathcal{G}\). Thus there exists a constant
\[
C_{0}= |V+a|_{L^{\infty}(\mathcal{G})} < \infty.
\]
Using that \(\sigma_{1},\sigma_{3}\) are unitary matrices, we obtain
\[
|Au|_{2}
\le |-i\sigma_{1}u'|_{2} + |(V+a)\sigma_{3}u|_{2}
\le |u'|_{2} + C_{0}|u|_{2},
\]
hence for some constant \(C_{1}>0\),
\[
|Au|_{2}^{2}
\le C_{1}^{2}\big(|u'|_{2}^{2} + |u|_{2}^{2}\big).
\]
Consequently,
\[
\|u\|_{\mathcal{D}(A)}^{2}
= |Au|_{2}^{2} + |u|_{2}^{2}
\le (C_{1}^{2}+1)\big(|u'|_{2}^{2} + |u|_{2}^{2}\big)
= (C_{1}^{2}+1)\,\|u\|_{ H^{1}(\mathcal{G},\mathbb{C}^{2})}^{2}.
\]

Conversely, from the identity
\[
-i\sigma_{1}u' = Au - (V+a)\sigma_{3}u
\]
we get
\[
|u'|_{2}
= |-i\sigma_{1}u'|_{2}
\le |Au|_{2} + |(V+a)\sigma_{3}u|_{2}
\le |Au|_{2} + C_{0}|u|_{2}.
\]
Using the elementary estimate \((\alpha+\beta)^{2}\le 2(\alpha^{2}+\beta^{2})\) we infer
\[
|u'|_{2}^{2}
\le 2|Au|_{2}^{2} + 2C_{0}^{2}|u|_{2}^{2}
\le C_{2}\big(|Au|_{2}^{2} + |u|_{2}^{2}\big)
= C_{2}\|u\|_{\mathcal{D}(A)}^{2}
\]
for some constant \(C_{2}>0\). Therefore,
\[
\|u\|_{ H^{1}(\mathcal{G},\mathbb{C}^{2})}^{2}
= |u'|_{2}^{2} + |u|_{2}^{2}
\le (C_{2}+1)\,\|u\|_{\mathcal{D}(A)}^{2}.
\]
Combining the two inequalities, there exist constants \(c_{1},c_{2}>0\) such that
\[
c_{1}\,\|u\|_{ H^{1}(\mathcal{G},\mathbb{C}^{2})}
\le \|u\|_{\mathcal{D}(A)}
\le c_{2}\,\|u\|_{ H^{1}(\mathcal{G},\mathbb{C}^{2})}
\quad\text{for all }u\in\mathcal{D}(A),
\]
which proves the equivalence of the two norms.
\end{proof}

	Now we consider the operator $A$. Let $\left(E_{\gamma}\right)_{\gamma \in \mathbb{R}}$ and
	$\left(F_{\gamma}\right)_{\gamma \ge 0}$ denote the spectral families of $A$ and $A^{2}$, respectively.
	Recall that
	\begin{equation}\label{eq-3.2}
		F_{\gamma}
		= E_{\gamma^{1/2}} - E_{-\gamma^{1/2}-0}
		= E_{[-\gamma^{1/2},\,\gamma^{1/2}]}
		\quad \text{for all } \gamma \ge 0.
	\end{equation}
Here $E_{\lambda-0}$ denotes the left limit of the spectral family, namely
\[
E_{\lambda-0}:=\lim_{\varepsilon\to 0} E_{\lambda-\varepsilon}
\quad\text{in the strong operator topology}.
\]

	\begin{Lem}\label{lem-3.3}
		If $(V_{1})$ holds, then
		\[
		\sigma(A)\subset(-\infty,-a]\cup[a,\infty)
		\quad\text{and}\quad
		\inf\sigma(|A|) \le a + \sup_{\mathcal{G}} V.
		\]
	\end{Lem}
	
	\begin{proof}
		Assume $(V_{1})$. By \eqref{eq-3.2} and Lemma~\ref{lem-3.1} we have
		\[
		\dim\big(E_{[-\gamma^{1/2},\,\gamma^{1/2}]}L^{2}\big)
		= \dim\big(F_{\gamma}L^{2}\big) = 0
		\quad\text{for } 0\le \gamma < a^{2},
		\]
		hence there is no spectral mass in $(-a,a)$ and therefore
		\[
		\sigma(A)\subset \mathbb{R}\setminus(-a,a)
		= (-\infty,-a]\cup[a,\infty).
		\]
		
		We now estimate the bottom of $\sigma(|A|)$. By $(V_{1})$ and the periodicity of $\mathcal{G}$ under the free, cocompact
		$\mathbb{Z}^{d}$-action introduced in the introduction, there exists a compact fundamental cell
		$\mathcal{K}\subset\mathcal{G}$ such that
		\[
		\mathcal{G} = \bigcup_{k\in\mathbb{Z}^{d}} T^{k}(\mathcal{K}),
		\]
		where $T^{k}\in\mathrm{Aut}(\mathcal{G})$ are graph automorphisms and the overlaps $T^{k}(\mathcal{K})\cap T^{\ell}(\mathcal{K})$
		consist only of boundary vertices when $k\ne\ell$.
		Since $V\in C^{1}(\mathcal{G})$ and is $\mathbb{Z}^{d}$-periodic, it is bounded on $\mathcal{G}$; set
		\[
		V_{\infty} = \sup_{\mathcal{G}} V < \infty.
		\]
		
		For $N\in\mathbb{N}$, let
		\[
		\Lambda_{N}= \big\{k\in\mathbb{Z}^{d} : \max_{1\le j\le d} |k_{j}|\le N\big\},
		\qquad
		\mathcal{G}_{N}= \bigcup_{k\in\Lambda_{N}} T^{k}(\mathcal{K}).
		\]
		Then $\mathcal{G}_{N}$ is a finite connected subgraph of $\mathcal{G}$, and
		\[
		|\Lambda_{N}|\sim (2N+1)^{d}\quad\text{as }N\to\infty.
		\]
		By standard cut-off constructions on periodic quantum graphs, there exists a family of functions
		$\eta_{N}\in H^{1}(\mathcal{G})$ such that
		\[
		0\le \eta_{N}\le 1 \text{ on }\mathcal{G},\quad
		\eta_{N}\equiv 1 \text{ on }\mathcal{G}_{N-1},\quad
		\mathrm{supp}\,\eta_{N} \subset \mathcal{G}_{N},
		\]
		and there exists a constant $C_{*}>0$, independent of $N$, such that
		\[
		\int_{\mathcal{G}} |\eta_{N}'|^{2}\,dx \le C_{*}\,\big(|\Lambda_{N}| - |\Lambda_{N-1}|\big)
		\le C\,N^{d-1}
		\]
		for some $C>0$ depending only on the geometry of the fundamental cell and the action.
		Moreover, there exists $c_{0}>0$ such that
		\[
		\int_{\mathcal{G}} |\eta_{N}|^{2}\,dx
		\ge \int_{\mathcal{G}_{N-1}} 1\,dx
		\ge c_{0}\,|\Lambda_{N-1}|
		\ge c_{1} N^{d}
		\]
		for some $c_{1}>0$ and all $N$ large.
		
		Define spinors
		\[
		u_{N}(x)= \binom{\eta_{N}(x)}{0} \in H^{1}(\mathcal{G},\mathbb{C}^{2}),
		\qquad
		v_{N}= \frac{u_{N}}{|u_{N}|_{2}}.
		\]
		By construction, $u_{N}\in\operatorname{dom}(A)$: the first component $\eta_{N}$ is continuous,
		the second component is identically zero, and the Kirchhoff-type condition \eqref{eq-2.4} is trivially satisfied
		since $u_{N}^{2}\equiv 0$.
		
		From the above estimates,
		\[
		|u_{N}'|_{2}^{2}
		= \int_{\mathcal{G}} |\eta_{N}'|^{2}\,dx
		\le C\,N^{d-1},
		\qquad
		|u_{N}|_{2}^{2}
		= \int_{\mathcal{G}} |\eta_{N}|^{2}\,dx
		\ge c_{1}N^{d},
		\]
		so
		\[
		|v_{N}'|_{2}^{2}
		= \frac{|u_{N}'|_{2}^{2}}{|u_{N}|_{2}^{2}}
		\le \frac{C\,N^{d-1}}{c_{1}N^{d}}
		= \frac{C'}{N}
		\rightarrow 0 \quad\text{as }N\to\infty.
		\]
		In particular, $|v_{N}'|_{2}\to 0$ as $N\to\infty$, while $|v_{N}|_{2}=1$ for all $N$.
		
		Using the definition of $A$ and the unitarity of $\sigma_{1},\sigma_{3}$ we obtain
		\[
		|A v_{N}|_{2}
		= \big|-i\sigma_{1}v_{N}' + (V+a)\sigma_{3}v_{N}\big|_{2}
		\le |v_{N}'|_{2} + |(V+a)\sigma_{3}v_{N}|_{2}
		\le |v_{N}'|_{2} + (a+V_{\infty})|v_{N}|_{2},
		\]
		so
		\[
		|A v_{N}|_{2}
		\le a + V_{\infty} + |v_{N}'|_{2}
		\rightarrow a + V_{\infty}
		\quad\text{as }N\to\infty.
		\]
        Since $|A|$ is a nonnegative self-adjoint operator, its spectral bottom admits the variational characterization
\[
\inf\sigma(|A|)
= \inf_{\substack{u\in\mathcal{D}(A)\\ |u|_{2}=1}} (|A|u,u)_{L^{2}}.
\]
Therefore, for every $N$,
\[
\inf\sigma(|A|)
\le (|A|v_{N},v_{N})_{L^{2}}
\le \||A|v_{N}\|_{2}\,|v_{N}|_{2}
= |A v_{N}|_{2},
\]
where we used Cauchy--Schwarz and the identity $\||A|v_{N}\|_{2}=|A v_{N}|_{2}$ for $v_{N}\in\mathcal{D}(A)$.
Hence
\[
\inf\sigma(|A|)
\le \limsup_{N\to\infty}|A v_{N}|_{2}
\le a + \sup_{\mathcal{G}} V.
\]
This proves the second assertion.
	\end{proof}
	
	Note that we have an orthogonal decomposition
	\[
	L^{2} = L^{-}\oplus L^{0}\oplus L^{+}, \qquad u = u^{-} + u^{0} + u^{+},
	\]
	such that $A$ is negative definite on $L^{-}$, positive definite on $L^{+}$, and vanishes on $L^{0}$.
	Clearly, $L^{0}=\{0\}$ under assumption $(V_{1})$, since in this case $0\notin\sigma(A)$ by Lemmas~\ref{lem-3.1} and~\ref{lem-3.3}.

    Let $Y=\mathcal{D}(|A|^{1/2})$ and regard it as a real Hilbert space endowed with the inner product
\[
(u,v):=\Re\big(|A|^{1/2}u,|A|^{1/2}v\big)_{L^{2}}
+ \Re (u^{0},v^{0})_{L^{2}},
\]
and the norm $\|u\|=(u,u)^{1/2}$. There is an induced orthogonal decomposition
	\[
	Y = Y^{-}\oplus Y^{0}\oplus Y^{+},
	\quad\text{where } Y^{\pm} = Y\cap L^{\pm},\; Y^{0} = Y\cap L^{0},
	\]
	which is orthogonal with respect to both $(\cdot,\cdot)_{L^{2}}$ and $(\cdot,\cdot)$.

    \begin{Lem}\label{lem-3.4}
If $(V_{1})$ holds, then $Y$ and $ H^{1/2}(\mathcal{G},\mathbb{C}^{2})$ have equivalent norms, and
\[
a\,|u|_{2}^{2} \le \|u\|^{2}\quad\text{for all }u\in Y.
\]
\end{Lem}
	
	\begin{proof}
		First, By Theorem~\ref{theo-2.1} and Proposition~\ref{pro-2.1} applied with $H=A$,
		the form domain of $A$ coincides with the real interpolation space
		\[
		Y = \mathcal{D}(|A|^{1/2})
		= [L^{2}(\mathcal{G},\mathbb{C}^{2}),\,\mathcal{D}(A)]_{1/2}.
		\]
		By Lemma~\ref{lem-3.2}, the graph norm on $\mathcal{D}(A)$ and the
$ H^{1}(\mathcal{G},\mathbb{C}^{2})$-norm are equivalent. Therefore
\[
Y \cong [L^{2}(\mathcal{G},\mathbb{C}^{2}),\, H^{1}(\mathcal{G},\mathbb{C}^{2})]_{1/2}
=  H^{1/2}(\mathcal{G},\mathbb{C}^{2}),
\]
with equivalence of norms.
		
		Next, Lemma~\ref{lem-3.1} implies $\sigma(A^{2})\subset[a^{2},\infty)$, hence
		$\sigma(|A|)\subset[a,\infty)$ and $|A|\ge aI$ in the sense of quadratic forms.
		Under $(V_{1})$ we also have $0\notin\sigma(A)$, so $L^{0}=\{0\}$ and
		the inner product on $Y$ simplifies to
		\[
		(u,v) = \big(|A|^{1/2}u,\,|A|^{1/2}v\big)_{L^{2}},
		\quad u,v\in Y.
		\]
		Thus, for every $u\in Y$,
		\[
		\|u\|^{2}
		= \big(|A|^{1/2}u,\,|A|^{1/2}u\big)_{L^{2}}
		= (|A|u,u)_{L^{2}}
		\ge a\,(u,u)_{L^{2}}
		= a\,|u|_{2}^{2},
		\]
		which gives the claimed coercivity.
	\end{proof}
	
	Next, we explain that the bound states coincide with the critical points of the $C^{1}$ action functional
$\Phi:Y\to\mathbb R$ defined by
\begin{equation}\label{eq-3.3}
\Phi(u)
= \frac{1}{2}\big(\|u^{+}\|^{2}-\|u^{-}\|^{2}\big)
+ \frac{\omega}{2}\int_{\mathcal{G}} |u|^{2}\,dx
- \int_{\mathcal{G}} F(x,u)\,dx,
\qquad u\in Y.
\end{equation}

\begin{Pro}\label{pro-3.1}
A spinor $u$ is a bound state of frequency $\omega$ of the NLDE \eqref{eq-1.5}
if and only if it is a critical point of $\Phi$.
\end{Pro}

\begin{proof}
If $u$ is a bound state of frequency $\omega$ of \eqref{eq-1.5}, then it satisfies the Euler--Lagrange equation
associated with \eqref{eq-3.3}, hence $u$ is a critical point of $\Phi$.

Conversely, assume that $u$ is a critical point of $\Phi$, namely $u\in Y$ and
\begin{equation}\label{eq-3.4}
\langle \Phi'(u),\varphi\rangle
= \Re\int_{\mathcal{G}}\langle u,(A+\omega)\varphi\rangle\,dx
- \Re\int_{\mathcal{G}} \langle F_{u}(x,u),\varphi\rangle\,dx
=0
\quad \forall\,\varphi\in \operatorname{dom}(A).
\end{equation}

Fix an edge $e\in E$ and take
\begin{equation}\label{eq-3.5}
\varphi = \binom{\varphi^{1}}{0}
\quad\text{with}\quad
0\neq \varphi^{1}\in C_{0}^{\infty}(I_{e}).
\end{equation}
Since $\varphi$ is supported in the interior of $I_{e}$, inserting \eqref{eq-3.5} into \eqref{eq-3.4}
and using the explicit expression of $\mathcal{D}$ on $I_{e}$ we obtain
\[
-i\int_{I_{e}} u_{e}^{2}\,(\overline{\varphi^{1}})'\,dx_{e}
= \int_{I_{e}}
\Big[(a+V(x)+\omega)u_{e}^{1}
- (F_{u}(x,u_{e}))^{1}\Big]\,
\overline{\varphi^{1}}\,dx_{e}.
\]
Hence the distributional derivative of $u_{e}^{2}$ belongs to $L^{2}(I_{e})$, so $u_{e}^{2}\in H^{1}(I_{e})$,
and the first component of \eqref{eq-1.5} holds on $I_{e}$.
Exchanging the role of $\varphi^{1}$ and $\varphi^{2}$ in \eqref{eq-3.5} and repeating the argument,
we see that $u_{e}^{1}\in H^{1}(I_{e})$ and that the second component of \eqref{eq-1.5} holds as well.
Since $e$ is arbitrary, we conclude that $u\in  H^{1}(\mathcal{G},\mathbb{C}^{2})$
and $u$ solves \eqref{eq-1.5} on each edge.

We now prove that $u$ fulfills the vertex conditions \eqref{eq-2.3} and \eqref{eq-2.4}.
Fix a vertex $v$ and choose
\[
\varphi = \binom{\varphi^{1}}{0}\in\operatorname{dom}(\mathcal{A})
\quad\text{with compact support,}\quad
\varphi^{1}(v)=1,
\]
and $\varphi^{1}(v')=0$ for every vertex $v'\neq v$ with $v'\in \operatorname{supp}\varphi^{1}$.
Integrating by parts in \eqref{eq-3.4} and using the fact that $u$ satisfies \eqref{eq-1.5}
in the interior of each edge, the bulk terms cancel and we are left with the boundary contribution at $v$:
\[
\sum_{e\succ v} \varphi_{e}^{1}(v)\,u_{e}^{2}(v)_{\pm} = 0.
\]
Since $\varphi^{1}(v)=1$, this implies
\[
\sum_{e\succ v} u_{e}^{2}(v)_{\pm} = 0,
\]
so $u^{2}$ satisfies \eqref{eq-2.4}.

Next, fix a vertex $v$ with degree at least $2$ and two edges $e_{1},e_{2}\succ v$.
Choose
\[
\varphi = \binom{0}{\varphi^{2}}\in\operatorname{dom}(\mathcal{A})
\]
with compact support in the star of $v$, such that
\[
\varphi_{e_{1}}^{2}(v)_{\pm} = -\varphi_{e_{2}}^{2}(v)_{\pm}\neq 0,
\qquad
\varphi_{e}^{2}(v)_{\pm}=0 \ \ \forall\, e\succ v,\ e\neq e_{1},e_{2}.
\]
Then integrating by parts in \eqref{eq-3.4} yields
\[
\varphi_{e_{1}}^{2}(v)_{\pm}\,u_{e_{1}}^{1}(v)
+ \varphi_{e_{2}}^{2}(v)_{\pm}\,u_{e_{2}}^{1}(v) = 0,
\]
hence $u_{e_{1}}^{1}(v)=u_{e_{2}}^{1}(v)$. Repeating the argument for any pair of edges incident at $v$
we obtain \eqref{eq-2.3}.

Therefore $u\in \operatorname{dom}(\mathcal{A})$ and satisfies \eqref{eq-1.5}, namely $u$ is a bound state
of frequency $\omega$ of the NLDE.
\end{proof}

In summary, by the orthogonal decomposition $Y = Y^{-}\oplus Y^{0}\oplus Y^{+}$, the action functional
\eqref{eq-3.3} can also be written in the form
\begin{equation}\label{eq-3.6}
\Phi(u)
= \frac{1}{2}\big(\|u^{+}\|^{2} - \|u^{-}\|^{2}\big)
+ \frac{\omega}{2}\int_{\mathcal{G}} |u|^{2}\,dx
- \Psi(u),
\end{equation}
where
\[
\Psi(u)= \int_{\mathcal{G}} F(x,u)\,dx .
\]

\subsection{Proof of Theorem \ref{theo-1.1}}
In this section we prove Theorem \ref{theo-1.1}. Recall the functional $\Phi$ defined on the space
$Y=Y^{-}\oplus Y^{+}$ by \eqref{eq-3.6}:
\[
\Phi(u)
= \frac{1}{2}\Big(\|u^{+}\|^{2}-\|u^{-}\|^{2} + \omega|u|_{2}^{2}\Big)
- \Psi(u),
\qquad
\Psi(u) = \int_{\mathcal{G}} F(x,u)\,dx.
\]
In order to apply the critical point theorems of Section~2, we set $M=Y^{-}$, $N=Y^{+}$ and fix a countable dense subset $\mathcal S\subset M^{*}$.
First, by $(\omega)$ and Lemma~\ref{lem-3.4} we have
\begin{equation}\label{eq-3.7}
\frac{a-|\omega|}{a}\,\|u^{+}\|^{2}
\le
\big(\|u^{+}\|^{2} \pm \omega|u^{+}|_{2}^{2}\big)
\le
\frac{a+|\omega|}{a}\,\|u^{+}\|^{2}
\end{equation}
and
\begin{equation}\label{eq-3.8}
\frac{a-|\omega|}{a}\,\|u^{-}\|^{2}
\le
\big(\|u^{-}\|^{2} \pm \omega|u^{-}|_{2}^{2}\big)
\le
\frac{a+|\omega|}{a}\,\|u^{-}\|^{2}.
\end{equation}

\begin{Lem}\label{lem-3.5}
Assume $(F_{0})$--$(F_{2})$ and $(F_{5})$. Then $\Psi$ is weakly sequentially lower semicontinuous on $Y$
and $\Phi':Y\to Y^{*}$ is weakly sequentially continuous. Moreover, for every $c>0$ there exists
$\zeta=\zeta(c)>0$ such that
\[
\|u\| < \zeta\,\|u^{+}\|
\quad\text{for all }u\in\Phi_{c}=\{u\in Y:\Phi(u)\ge c\}.
\]
\end{Lem}

\begin{proof}
This follows from Lemma~4.1 in \cite{yang2025bound}: the proof uses only the $\mathbb Z^{d}$-periodic structure, local compact embeddings on finite unions of translates of a fundamental cell, and a tail estimate based on the nonnegativity of $F$. Hence we omit it.
\end{proof}

	\begin{Lem}\label{lem-3.6}
There exists $\rho>0$ such that
\[
\eta= \inf\big\{\Phi(u) : u\in Y^{+},\ \|u\|=\rho\big\} > 0.
\]
\end{Lem}

\begin{proof}
Set $q=2+\nu\in(2,3)$, where $\nu\in(0,1)$ is given by $(F_{5})$.
Using $(F_{0})$, $(F_{2})$ and $(F_{5})$, one checks that for every $\varepsilon>0$
there exists $C_{\varepsilon}>0$ such that
\[
F(x,u) \le \varepsilon|u|^{2} + C_{\varepsilon}|u|^{q}
\quad\text{for all }(x,u)\in\mathcal{G}\times\mathbb{C}^{2}.
\]
Hence, for all $u\in Y$,
\begin{equation}\label{eq-3.10}
\begin{aligned}
\Psi(u)
&= \int_{\mathcal{G}}F(x,u)\,dx \\
&\le \varepsilon\int_{\mathcal{G}}|u|^{2}\,dx
+ C_{\varepsilon}\int_{\mathcal{G}}|u|^{q}\,dx \\
&\le \frac{\varepsilon}{a}\,\|u\|^{2} + C_{\varepsilon}C_{q}^{\,q}\,\|u\|^{q},
\end{aligned}
\end{equation}
where we used $a|u|_{2}^{2}\le \|u\|^{2}$ (Lemma~\ref{lem-3.4}) and the continuous embedding
$Y\hookrightarrow L^{q}(\mathcal{G},\mathbb{C}^{2})$, i.e. $|u|_{q}\le C_{q}\|u\|$.

Now let $u\in Y^{+}$. Using \eqref{eq-3.7} (with $u^{-}=0$) and \eqref{eq-3.10}, we obtain
\[
\Phi(u)
= \frac{1}{2}\big(\|u\|^{2} + \omega|u|_{2}^{2}\big) - \Psi(u)
\ge \frac{a-|\omega|}{2a}\,\|u\|^{2}
- \frac{\varepsilon}{a}\,\|u\|^{2} - C_{\varepsilon}C_{q}^{\,q}\,\|u\|^{q}.
\]
Choose $\varepsilon>0$ so small that
\[
\delta=\frac{a-|\omega|}{2a}-\frac{\varepsilon}{a}>0,
\]
and then choose $\rho>0$ such that
\[
C_{\varepsilon}C_{q}^{\,q}\,\rho^{\,q-2}\le \frac{\delta}{2}.
\]
For all $u\in Y^{+}$ with $\|u\|=\rho$ we then have
\[
\Phi(u)\ge \delta\rho^{2}-C_{\varepsilon}C_{q}^{\,q}\rho^{q}
\ge \frac{\delta}{2}\,\rho^{2}>0.
\]
Thus $\eta= \inf\{\Phi(u):u\in Y^{+},\ \|u\|=\rho\}\ge \frac{\delta}{2}\rho^{2}>0$.
\end{proof}
	
	As a consequence of Lemmas \ref{lem-3.1} and \ref{lem-3.3} we know that
	\[
	a \le \inf \sigma(|A|) \le a + \sup_{\mathcal{G}} V .
	\]
	By $(F_{3})$ we also have
	\[
	\inf_{\mathcal{G}} b > \sup_{\mathcal{G}} V + a + \omega ,
	\]
	so we can choose a number $\gamma$ such that
	\begin{equation}\label{eq-3.11}
		a + \sup_{\mathcal{G}} V < \gamma < \inf_{\mathcal{G}} b - \omega .
	\end{equation}
	Since $A$ is invariant under the free, cocompact $\mathbb{Z}^{d}$--action on $\mathcal{G}$,
it is a $\mathbb{Z}^{d}$--periodic operator in the sense of Floquet theory. In particular,
$\sigma(A)$ has the usual band--gap structure (and may also contain flat bands).
Since $\gamma>\inf\sigma(|A|)$, the interval $(0,\gamma]$ intersects $\sigma(|A|)$, and hence the
spectral subspace
\[
Y_{0}:=\big(E_{\gamma}-E_{0}\big)L^{2}(\mathcal{G},\mathbb{C}^{2})
\]
is infinite-dimensional.

Moreover, for every $u\in Y_{0}$ the spectral support of $u$ is contained in $(0,\gamma]$ on the
positive part, so by the spectral calculus for $|A|$ we have
\begin{equation}\label{eq-3.12}
(a+\omega)\,|u|_{2}^{2}
\le \|u\|^{2} + \omega|u|_{2}^{2}
\le (\gamma+\omega)\,|u|_{2}^{2}.
\end{equation}

Choose an $L^{2}$-orthonormal sequence $\{e_{n}\}_{n\in\mathbb{N}}\subset Y_{0}$, i.e.
\[
(e_{n},e_{m})_{L^{2}}=\delta_{nm}\quad\text{and}\quad |e_{n}|_{2}=1 \ \text{for all }n.
\]
For each $n\in\mathbb{N}$, define
\[
Y_{n}= \operatorname{span}\{e_{1},\dots,e_{n}\},\qquad
E_{n}= Y^{-} \oplus Y_{n}.
\]

	\begin{Lem}\label{lem-3.7}
		For each $n\in\mathbb{N}$ one has $\sup \Phi(E_{n})<\infty$. Moreover, there exists a sequence $R_{n}>0$ such that
		\[
		\sup \Phi\big(E_{n}\setminus B_{n}\big) < \eta,
		\]
		where $B_{n}=\{u\in E_{n}:\|u\|\le R_{n}\}$ and
		\[
		\eta= \inf\big\{\Phi(u):u\in Y^{+},\ \|u\|=\rho\big\} > 0
		\]
		is given by Lemma~\ref{lem-3.6}.
	\end{Lem}
	\begin{proof}
For $u\in E_{n}$, since $F\ge 0$ on $\mathcal{G}\times\mathbb{C}^{2}$, we have
\begin{equation}\label{eq-3.13}
\begin{aligned}
\Phi(u)
&= \frac{1}{2}\|u^{+}\|^{2}
- \frac{1}{2}\|u^{-}\|^{2}
+ \frac{\omega}{2}|u|_{2}^{2}
- \int_{\mathcal{G}} F(x,u)\,dx \\
&\le \frac{1}{2}\|u^{+}\|^{2}
+ \frac{|\omega|}{2}\,|u|_{2}^{2}
\le \frac{1}{2}\|u^{+}\|^{2}
+ \frac{|\omega|}{2a}\,\|u\|^{2},
\end{aligned}
\end{equation}
where in the last inequality we used Lemma~\ref{lem-3.4}. In particular, $\Phi$ is well defined on $E_n$.

Fix $n\in\mathbb{N}$. We claim that $\Phi(u)\to -\infty$ as $\|u\|\to\infty$ with $u\in E_{n}$.
Assume by contradiction that there exist $M>0$ and a sequence $(u_{j})\subset E_{n}$ such that
\[
\|u_{j}\|\to\infty
\quad\text{and}\quad
\Phi(u_{j}) \ge -M \quad\text{for all }j.
\]
Set $v_{j}=u_{j}/\|u_{j}\|$. Then $(v_j)$ is bounded in $Y$ and $\|v_j\|=1$ for all $j$.
Writing $v_j=v_j^-+v_j^+$ with $v_j^-\in Y^{-}$ and $v_j^+\in Y_n$, up to a subsequence,
\[
v_{j} \rightharpoonup v \ \text{in }Y,\qquad
v_{j}^{-} \rightharpoonup v^{-}\ \text{in }Y^{-},\qquad
v_{j}^{+}\to v^{+}\ \text{in }Y_{n}.
\]

Dividing the identity defining $\Phi(u_j)$ by $\|u_j\|^{2}$ we obtain
\begin{equation}\label{eq-3.14}
\frac{\Phi(u_{j})}{\|u_{j}\|^{2}}
= \frac{1}{2}\big(\|v_{j}^{+}\|^{2}
- \|v_{j}^{-}\|^{2}
+ \omega|v_{j}|_{2}^{2}\big)
- \int_{\mathcal{G}} \frac{F(x,u_{j})}{\|u_{j}\|^{2}}\,dx
\ge -\frac{M}{\|u_{j}\|^{2}} = o(1).
\end{equation}
Since $F\ge 0$, \eqref{eq-3.14} yields
\[
\frac{1}{2}\big(\|v_{j}^{+}\|^{2}
- \|v_{j}^{-}\|^{2}
+ \omega|v_{j}|_{2}^{2}\big)\ge o(1).
\]
Using Lemma~\ref{lem-3.4} and \eqref{eq-3.8}, we get
\begin{equation}\label{eq-3.15}
\begin{aligned}
o(1)
&\le \frac{1}{2}\big(\|v_{j}^{+}\|^{2}
- \|v_{j}^{-}\|^{2}
+ \omega|v_{j}|_{2}^{2}\big) \\
&= \|v_{j}^{+}\|^{2}
- \frac{1}{2}\|v_{j}\|^{2}
+ \frac{\omega}{2}|v_{j}|_{2}^{2} \\
&\le \|v_{j}^{+}\|^{2}
- \frac{1}{2}\|v_{j}\|^{2}
+ \frac{|\omega|}{2a}\|v_{j}\|^{2},
\end{aligned}
\end{equation}
and therefore
\[
\|v_{j}^{+}\|^{2}
\ge \frac{a-|\omega|}{2a}\,\|v_{j}\|^{2} + o(1).
\]
In particular, $\liminf_{j\to\infty}\|v_{j}^{+}\|>0$, hence $v^{+}\neq 0$.

Next we prepare the asymptotically linear decomposition. By $(F_{3})$ and $(F_{2})$ one has
$|F_{u}(x,u)|\le C|u|$ for all $(x,u)\in\mathcal{G}\times\mathbb{C}^{2}$ for some $C>0$,
hence
\[
0\le F(x,u)=\int_{0}^{1}F_{u}(x,tu)\cdot u\,dt \le C|u|^{2}.
\]
Set
\[
R(x,u)=F(x,u)-\frac{1}{2}b(x)|u|^{2},
\qquad
b_{0}=\inf_{\mathcal{G}} b.
\]
Then $b$ is bounded (periodic on a cocompact graph), hence $|R(x,u)|\le C_{R}|u|^{2}$ for all $(x,u)$.
Moreover, by $(F_{3})$ we have
\[
\frac{R(x,u)}{|u|^{2}}\to 0\quad\text{as }|u|\to\infty,
\]
uniformly in $x\in\mathcal{G}$.

We now choose a compact set capturing almost all $L^{2}$-mass of $v$.
For $k\in\mathbb{Z}^{d}$ denote $\mathcal{K}_{k}:=T^{k}(\mathcal{K})$. Up to measure zero overlaps,
\[
|v|_{2}^{2}=\sum_{k\in\mathbb{Z}^{d}}\int_{\mathcal{K}_{k}}|v|^{2}\,dx.
\]
Since the series converges, there exists a finite set $\Lambda\subset\mathbb{Z}^{d}$ such that, setting
\[
\mathcal{K}_{\Lambda}:=\bigcup_{k\in\Lambda}\mathcal{K}_{k},
\]
we have
\[
\int_{\mathcal{K}_{\Lambda}}|v|^{2}\,dx \ge |v|_{2}^{2}-\varepsilon_{0},
\]
where $\varepsilon_{0}>0$ will be chosen below. Since $\mathcal{K}_{\Lambda}$ is a compact subgraph,
the embedding $Y\hookrightarrow L^{2}(\mathcal{K}_{\Lambda},\mathbb{C}^{2})$ is compact; thus, up to a subsequence,
\[
v_{j}\to v \quad\text{in }L^{2}(\mathcal{K}_{\Lambda},\mathbb{C}^{2}).
\]

For every $u\in Y$, using $F\ge 0$ we may restrict the nonlinear term:
\[
\Phi(u)
= \frac{1}{2}\|u^{+}\|^{2}
- \frac{1}{2}\|u^{-}\|^{2}
+ \frac{\omega}{2}|u|_{2}^{2}
- \int_{\mathcal{G}}F(x,u)\,dx
\le \frac{1}{2}\|u^{+}\|^{2}
- \frac{1}{2}\|u^{-}\|^{2}
+ \frac{\omega}{2}|u|_{2}^{2}
- \int_{\mathcal{K}_{\Lambda}}F(x,u)\,dx.
\]
Hence, for every $u\in Y$,
\begin{equation}\label{eq-3.16}
\begin{aligned}
\Phi(u)
&\le \frac{1}{2}\|u^{+}\|^{2}
- \frac{1}{2}\|u^{-}\|^{2}
+ \frac{\omega}{2}|u|_{2}^{2}
- \frac{1}{2}\int_{\mathcal{K}_{\Lambda}} b(x)|u|^{2}\,dx
- \int_{\mathcal{K}_{\Lambda}} R(x,u)\,dx \\
&\le \frac{1}{2}\big(\|u^{+}\|^{2}+\omega|u^{+}|_{2}^{2}\big)
- \frac{a-|\omega|}{2a}\,\|u^{-}\|^{2}
- \frac{b_{0}}{2}\int_{\mathcal{K}_{\Lambda}}|u|^{2}\,dx
- \int_{\mathcal{K}_{\Lambda}} R(x,u)\,dx,
\end{aligned}
\end{equation}
where we used \eqref{eq-3.8} and $b(x)\ge b_{0}$.

Now evaluate the quadratic part at the limit $v=v^{-}+v^{+}$. Since $v^{+}\in Y_{n}\subset Y_{0}$,
by \eqref{eq-3.12} we have $\|v^{+}\|^{2}+\omega|v^{+}|_{2}^{2}\le (\gamma+\omega)|v^{+}|_{2}^{2}$.
Using also $b_{0}-\gamma-\omega>0$ from \eqref{eq-3.11}, we obtain
\[
\big(\|v^{+}\|^{2}+\omega|v^{+}|_{2}^{2}\big)
- \frac{a-|\omega|}{a}\,\|v^{-}\|^{2}
- b_{0}|v|_{2}^{2}
\le -\big(b_{0}-\gamma-\omega\big)|v^{+}|_{2}^{2}
- \frac{a-|\omega|}{a}\,\|v^{-}\|^{2}<0.
\]
Choose $\varepsilon_{0}>0$ so small that
\begin{equation}\label{eq-3.17}
\big(\|v^{+}\|^{2}+\omega|v^{+}|_{2}^{2}\big)
- \frac{a-|\omega|}{a}\,\|v^{-}\|^{2}
- b_{0}\int_{\mathcal{K}_{\Lambda}}|v|^{2}\,dx<0,
\end{equation}
which is possible because $\int_{\mathcal{K}_{\Lambda}}|v|^{2}\,dx \ge |v|_{2}^{2}-\varepsilon_{0}$.

Next we show that
\[
\int_{\mathcal{K}_{\Lambda}}\frac{R(x,u_{j})}{\|u_{j}\|^{2}}\,dx \to 0.
\]
Indeed, fix $\varepsilon>0$. By the uniform limit $R(x,u)/|u|^{2}\to 0$ as $|u|\to\infty$,
there exists $R_{\varepsilon}>0$ such that $|R(x,u)|\le \varepsilon |u|^{2}$ for all $|u|\ge R_{\varepsilon}$
and all $x\in\mathcal{G}$. Then, using also $|R(x,u)|\le C_{R}|u|^{2}$, we get
\[
\begin{aligned}
\int_{\mathcal{K}_{\Lambda}}\left|\frac{R(x,u_{j})}{\|u_{j}\|^{2}}\right|dx
&\le \varepsilon\int_{\mathcal{K}_{\Lambda}}|v_{j}|^{2}\,dx
+ C_{R}\int_{\mathcal{K}_{\Lambda}\cap\{|u_{j}|<R_{\varepsilon}\}}\frac{|u_{j}|^{2}}{\|u_{j}\|^{2}}\,dx \\
&\le \varepsilon\int_{\mathcal{K}_{\Lambda}}|v_{j}|^{2}\,dx
+ C_{R}\frac{R_{\varepsilon}^{2}}{\|u_{j}\|^{2}}\,|\mathcal{K}_{\Lambda}|.
\end{aligned}
\]
Letting $j\to\infty$ yields
\[
\limsup_{j\to\infty}\int_{\mathcal{K}_{\Lambda}}\left|\frac{R(x,u_{j})}{\|u_{j}\|^{2}}\right|dx
\le \varepsilon\int_{\mathcal{K}_{\Lambda}}|v|^{2}\,dx.
\]
Since $\varepsilon>0$ is arbitrary, this proves the claim.

Finally, divide \eqref{eq-3.16} (applied to $u=u_j$) by $\|u_j\|^{2}$ and use $u_j=\|u_j\|v_j$:
\[
\frac{\Phi(u_{j})}{\|u_{j}\|^{2}}
\le \frac{1}{2}\big(\|v_{j}^{+}\|^{2}+\omega|v_{j}^{+}|_{2}^{2}\big)
- \frac{a-|\omega|}{2a}\,\|v_{j}^{-}\|^{2}
- \frac{b_{0}}{2}\int_{\mathcal{K}_{\Lambda}}|v_{j}|^{2}\,dx
- \int_{\mathcal{K}_{\Lambda}}\frac{R(x,u_{j})}{\|u_{j}\|^{2}}\,dx.
\]
Taking $\liminf$ on the left-hand side and $\limsup$ on the right-hand side, using
$v_j^+\to v^+$ in $Y_n$, $v_j\to v$ in $L^{2}(\mathcal{K}_{\Lambda})$, and the fact that the last integral
tends to $0$, we obtain
\[
0\le \liminf_{j\to\infty}\frac{\Phi(u_{j})}{\|u_{j}\|^{2}}
\le \frac{1}{2}\big(\|v^{+}\|^{2}+\omega|v^{+}|_{2}^{2}\big)
- \frac{a-|\omega|}{2a}\,\|v^{-}\|^{2}
- \frac{b_{0}}{2}\int_{\mathcal{K}_{\Lambda}}|v|^{2}\,dx < 0,
\]
where the last strict inequality follows from \eqref{eq-3.17}. This is a contradiction. Hence
\[
\Phi(u)\to -\infty \quad\text{as }\|u\|\to\infty,\ u\in E_{n},
\]
and in particular $\sup\Phi(E_{n})<\infty$.

Finally, since $\Phi(u)\to -\infty$ as $\|u\|\to\infty$ on $E_{n}$, we have
\[
\lim_{R\to\infty}\ \sup\{\Phi(u):u\in E_{n},\,\|u\|\ge R\} = -\infty.
\]
Recalling that $\eta>0$ from Lemma~\ref{lem-3.6}, we can choose $R_{n}>0$ large enough so that
\[
\sup\{\Phi(u):u\in E_{n},\,\|u\|\ge R_{n}\} < \eta.
\]
Equivalently, setting $B_{n}:=\{u\in E_{n}:\|u\|\le R_{n}\}$, we obtain
\[
\sup \Phi(E_{n}\setminus B_{n}) < \eta,
\]
as claimed.
\end{proof}

	As a consequence, we have:
\begin{Lem}\label{lem-3.8}
	We have $\Phi(u)\le 0$ for all $u\in\partial Q$, where
	\[
	Q= \{u = u^{-} + s e_{1} : u^{-}\in Y^{-},\ s\ge 0,\ \|u\|\le R_{1}\}.
	\]
\end{Lem}

\begin{proof}
Since $F\ge0$, we have $\Psi(u)\ge0$ for all $u\in Y$.
If $u\in Y^{-}$, then $u^{+}=0$ and hence
\[
\Phi(u)= -\frac12\|u\|^{2}+\frac{\omega}{2}|u|_{2}^{2}-\Psi(u)
\le -\frac12\big(\|u\|^{2}-\omega|u|_{2}^{2}\big)
\le -\frac{a-|\omega|}{2a}\|u\|^{2}\le0,
\]
where we used $|u|_{2}^{2}\le \frac1a\|u\|^{2}$ from Lemma~\ref{lem-3.4}.
Thus $\Phi\le0$ on $Y^{-}$, in particular on $\{u\in Q:\ s=0\}$.

Now let $E_{1}=Y^{-}\oplus Y_{1}$.
By Lemma~\ref{lem-3.7} (with $n=1$), $\Phi(u)\to-\infty$ as $\|u\|\to\infty$ with $u\in E_{1}$.
Therefore
\[
\lim_{R\to\infty}\sup\{\Phi(u):u\in E_{1},\ \|u\|=R\}=-\infty,
\]
and we can choose $R_{1}>0$ such that
\[
\sup\{\Phi(u):u\in E_{1},\ \|u\|=R_{1}\}\le0.
\]
If $u\in\partial Q$ with $s>0$, then $u\in E_{1}$ and $\|u\|=R_{1}$, hence $\Phi(u)\le0$.
Combining the two parts we conclude that $\Phi(u)\le0$ for all $u\in\partial Q$.
\end{proof}

\begin{Lem}\label{lem-3.9}
If $(\omega)$, $(F_{0})$--$(F_{2})$, $(F_{3})$ and $(F_{4})$ hold, then any $(C)_{c}$-sequence is bounded in $Y$.
\end{Lem}

\begin{proof}
Let $(u_{j})\subset Y$ be a $(C)_{c}$-sequence, namely
\[
\Phi(u_{j}) \to c
\quad\text{and}\quad
\big(1+\|u_{j}\|\big)\,\|\Phi^{\prime}(u_{j})\|_{Y^{*}} \to 0 .
\]
Recall that
\[
\Phi(u)-\frac12\langle \Phi'(u),u\rangle
=\int_{\mathcal{G}} \hat F(x,u)\,dx,
\qquad
\hat F(x,u)=\frac12\,F_{u}(x,u)\cdot u - F(x,u).
\]
By $(F_{4})$ we have $\hat F(x,u)\ge 0$. Hence, for $j$ large, there exists $C>0$ such that
\begin{equation}\label{eq-3.18}
0 \le \int_{\mathcal{G}} \hat F(x,u_{j})\,dx
=\Phi(u_{j})-\frac12\langle \Phi'(u_{j}),u_{j}\rangle
\le C .
\end{equation}

Assume by contradiction that $\|u_{j}\|\to\infty$ and set
\[
w_{j}=\frac{u_{j}}{\|u_{j}\|}.
\]
Then $\|w_{j}\|=1$ for all $j$. By Lemma~\ref{lem-3.4} and the embedding
$Y\hookrightarrow L^{s}(\mathcal{G},\mathbb{C}^{2})$ for every $s\in[2,\infty)$, we have
\[
|w_{j}|_{s}\le C_{s}\quad\text{for all }s\in[2,\infty),
\]
with constants $C_{s}$ independent of $j$.

Fix $\kappa\in(0,2)$ and $R>0$ given by $(F_{4})$, so that
\[
\hat F(x,u)\ge c_{1}|u|^{\kappa}
\quad\text{for all }x\in\mathcal{G},\ |u|\ge R.
\]
We claim that for every $\rho>0$ there exists $a_{\rho}>0$ such that
\begin{equation}\label{eq-3.19}
\hat F(x,u)\ge a_{\rho}|u|^{\kappa}
\quad\text{for all }x\in\mathcal{G},\ |u|\ge \rho.
\end{equation}
Indeed, if $\rho\ge R$ then \eqref{eq-3.19} holds with $a_{\rho}=c_{1}$.
If $0<\rho<R$, use the $\mathbb{Z}^{d}$-periodicity of $F$ and the cocompactness to reduce $x$ to a fixed
fundamental cell $\mathcal{K}$: the set
\[
\mathcal{C}_{\rho}=\big\{(x,u)\in \mathcal{K}\times\mathbb{C}^{2}:\ \rho\le |u|\le R\big\}
\]
is compact, and $(F_{4})$ implies $\hat F(x,u)>0$ for $u\neq 0$, hence
\[
m_{\rho}=\min_{\mathcal{C}_{\rho}}\hat F>0.
\]
Therefore, for $\rho\le |u|\le R$ and $x\in\mathcal{G}$ we have $\hat F(x,u)\ge m_{\rho}\ge \frac{m_{\rho}}{R^{\kappa}}|u|^{\kappa}$,
while for $|u|\ge R$ we have $\hat F(x,u)\ge c_{1}|u|^{\kappa}$. Thus \eqref{eq-3.19} holds with
\[
a_{\rho}=\min\left\{c_{1},\frac{m_{\rho}}{R^{\kappa}}\right\}>0.
\]

For $\rho>0$ set
\[
I_{j}(\rho)=\big\{x\in\mathcal{G}:\ |u_{j}(x)|\ge\rho\big\},
\qquad
I_{j}^{c}(\rho)=\mathcal{G}\setminus I_{j}(\rho).
\]
Using \eqref{eq-3.18} and \eqref{eq-3.19}, we obtain
\[
\int_{I_{j}(\rho)} |w_{j}|^{\kappa}\,dx
=\int_{I_{j}(\rho)} \frac{|u_{j}|^{\kappa}}{\|u_{j}\|^{\kappa}}\,dx
\le \int_{I_{j}(\rho)} \frac{\hat F(x,u_{j})}{a_{\rho}\,\|u_{j}\|^{\kappa}}\,dx
\le \frac{C}{a_{\rho}\,\|u_{j}\|^{\kappa}}
\to 0.
\]
Let $s\in(\kappa,\infty)$ and choose $\bar s\in(\max\{2,s\},\infty)$. By Hölder interpolation,
\begin{equation}\label{eq-3.20}
\begin{aligned}
\int_{I_{j}(\rho)}|w_{j}|^{s}\,dx
&\le
\left(\int_{I_{j}(\rho)}|w_{j}|^{\kappa}\,dx\right)^{\frac{\bar s-s}{\bar s-\kappa}}
\left(\int_{I_{j}(\rho)}|w_{j}|^{\bar s}\,dx\right)^{\frac{s-\kappa}{\bar s-\kappa}}
\to 0,
\end{aligned}
\end{equation}
since the first factor tends to $0$ and the second is uniformly bounded.
In particular, taking $s=2>\kappa$ yields
\[
\int_{I_{j}(\rho)}|w_{j}|^{2}\,dx\to 0
\quad\text{for every }\rho>0.
\]

Next, by $(F_{2})$ we know that for any $\varepsilon>0$ there exists $\rho_{\varepsilon}>0$ such that
\[
|F_{u}(x,u)|\le \varepsilon |u|
\quad\text{for all }x\in\mathcal{G},\ |u|\le \rho_{\varepsilon}.
\]
By $(F_{3})$ and the continuity of $F_{u}$, there exists $C_{0}>0$ such that
\[
|F_{u}(x,u)|\le C_{0}|u|
\quad\text{for all }(x,u)\in\mathcal{G}\times\mathbb{C}^{2}.
\]

Now test the Cerami condition in the direction $u_{j}^{+}-u_{j}^{-}$. From the expression of $\Phi'$ we have
\[
\langle \Phi^{\prime}(u_{j}),u_{j}^{+}-u_{j}^{-}\rangle
= \|u_{j}\|^{2}
+ \omega\big(|u_{j}^{+}|_{2}^{2}-|u_{j}^{-}|_{2}^{2}\big)
- \int_{\mathcal{G}} F_{u}(x,u_{j})\cdot (u_{j}^{+}-u_{j}^{-})\,dx .
\]
Since $\big(1+\|u_{j}\|\big)\|\Phi^{\prime}(u_{j})\|_{Y^{*}}\to 0$, dividing by $\|u_{j}\|^{2}$ gives
\[
\frac{\langle \Phi^{\prime}(u_{j}),u_{j}^{+}-u_{j}^{-}\rangle}{\|u_{j}\|^{2}}=o(1).
\]
Using Lemma~\ref{lem-3.4} we have $|u_{j}|_{2}^{2}\le \frac{1}{a}\|u_{j}\|^{2}$, hence
\[
1+\omega\frac{|u_{j}^{+}|_{2}^{2}-|u_{j}^{-}|_{2}^{2}}{\|u_{j}\|^{2}}
\ge 1-\frac{|\omega|}{a}.
\]
Therefore,
\[
1-\frac{|\omega|}{a}
\le o(1)+\frac{1}{\|u_{j}\|^{2}}
\int_{\mathcal{G}} |F_{u}(x,u_{j})|\,|u_{j}^{+}-u_{j}^{-}|\,dx.
\]
Write $u_{j}=\|u_{j}\|w_{j}$ and note that $|u_{j}^{+}-u_{j}^{-}|=\|u_{j}\||w_{j}^{+}-w_{j}^{-}|$. Defining
$\frac{|F_{u}(x,u_{j})|}{|u_{j}|}=0$ on $\{u_{j}=0\}$, we get
\[
1-\frac{|\omega|}{a}
\le o(1)+\int_{\mathcal{G}}\frac{|F_{u}(x,u_{j})|}{|u_{j}|}\,|w_{j}|\,|w_{j}^{+}-w_{j}^{-}|\,dx.
\]
Split the integral as $\int_{\mathcal{G}}=\int_{I_{j}^{c}(\rho_{\varepsilon})}+\int_{I_{j}(\rho_{\varepsilon})}$.
On $I_{j}^{c}(\rho_{\varepsilon})$ we have $|u_{j}|\le \rho_{\varepsilon}$ and hence
$\frac{|F_{u}(x,u_{j})|}{|u_{j}|}\le \varepsilon$. Thus
\[
\int_{I_{j}^{c}(\rho_{\varepsilon})}\frac{|F_{u}(x,u_{j})|}{|u_{j}|}\,|w_{j}|\,|w_{j}^{+}-w_{j}^{-}|
\le \varepsilon\,|w_{j}|_{2}\,|w_{j}^{+}-w_{j}^{-}|_{2}
\le C\,\varepsilon,
\]
where we used $|w_{j}|_{2}\le \frac{1}{\sqrt a}\|w_{j}\|=\frac{1}{\sqrt a}$ and
$|w_{j}^{+}-w_{j}^{-}|_{2}\le |w_{j}|_{2}$.

On $I_{j}(\rho_{\varepsilon})$ we use $\frac{|F_{u}(x,u_{j})|}{|u_{j}|}\le C_{0}$ to obtain
\[
\int_{I_{j}(\rho_{\varepsilon})}\frac{|F_{u}(x,u_{j})|}{|u_{j}|}\,|w_{j}|\,|w_{j}^{+}-w_{j}^{-}|
\le C_{0}\int_{I_{j}(\rho_{\varepsilon})}|w_{j}|\,|w_{j}^{+}-w_{j}^{-}|
\le C_{0}\left(\int_{I_{j}(\rho_{\varepsilon})}|w_{j}|^{2}\right)^{1/2}|w_{j}^{+}-w_{j}^{-}|_{2}\to 0,
\]
since $\int_{I_{j}(\rho_{\varepsilon})}|w_{j}|^{2}\to 0$ and $|w_{j}^{+}-w_{j}^{-}|_{2}$ is bounded.

Putting the two parts together yields
\[
1-\frac{|\omega|}{a}\le o(1)+C\,\varepsilon.
\]
Letting $j\to\infty$ and then $\varepsilon\to 0$ gives $1-\frac{|\omega|}{a}\le 0$, contradicting $(\omega)$.
Hence $(u_{j})$ must be bounded in $Y$.
\end{proof}

	Let $K=\{u \in Y \setminus\{0\} : \Phi^{\prime}(u)=0\}$ be the set of nontrivial critical points. To prove condition $(\Phi_{5})$, we argue by contradiction and assume that
	\begin{equation}\label{eq-3.21}
		K/\mathbb{Z}^{d} \text{ is a finite set.}
	\end{equation}
	Under this assumption we will show that $(\Phi_{5})$ holds; then, by Theorem~\ref{theo-2.3}, $\Phi$ possesses an unbounded sequence of positive critical values, which contradicts \eqref{eq-3.21}. Thus, in the sequel we work under the hypothesis \eqref{eq-3.21}.
	
	Let $\mathcal{F}$ be a set consisting of arbitrarily chosen representatives of the $\mathbb{Z}^{d}$-orbits of $K$. Then $\mathcal{F}$ is a finite set by \eqref{eq-3.21}, and since $\Phi^{\prime}$ is odd, we may assume $\mathcal{F}=-\mathcal{F}$. If $u \neq 0$ is a critical point of $\Phi$ and $\Omega \subset \mathcal{K}$, then using $(F_4)$ we have
	\[
	\Phi(u)=\Phi(u)-\frac{1}{2} \Phi^{\prime}(u)u
	=\int_{\mathcal{K}}\hat{F}(x,u)\,dx
	\geqslant \int_{\Omega} \hat{F}(x, u)\,dx>0.
	\]
	It follows that there exist $0<\theta \leqslant \vartheta$ such that
	\begin{equation}\label{eq-3.22}
		\theta<\min_{\mathcal{F}} \Phi=\min_{K} \Phi \leqslant \max_{K} \Phi=\max_{\mathcal{F}} \Phi<\vartheta .
	\end{equation}
	Let $[r]$ denote the integer part of $r \in \mathbb{R}$.

    \begin{Lem}\label{lem-3.10}
Assume \eqref{eq-3.21} holds and let $(u_m)\subset Y$ be a $(C)_c$--sequence.
Then either
\begin{itemize}
\item[(i)] $u_m\to 0$ in $Y$ and $c=0$, or
\item[(ii)] $c\ge \theta$ and there exist an integer $\ell\le [c/\theta]$, points
$\bar u_1,\dots,\bar u_\ell\in\mathcal F$, a subsequence (still denoted $(u_m)$), and sequences
$(a_m^i)\subset\mathbb Z^d$ such that
\[
\Big\|u_m-\sum_{i=1}^\ell (a_m^i*\bar u_i)\Big\|\to 0,
\qquad
\sum_{i=1}^\ell \Phi(\bar u_i)=c.
\]
\end{itemize}
\end{Lem}

\begin{proof}
Let $(u_m)\subset Y$ be a $(C)_c$--sequence, namely
\[
\Phi(u_m)\to c,
\qquad
(1+\|u_m\|)\|\Phi'(u_m)\|_{Y^*}\to 0.
\]
By Lemma~\ref{lem-3.9}, $(u_m)$ is bounded in $Y$. Using
\[
\Phi(u)-\frac12\langle \Phi'(u),u\rangle=\int_{\mathcal G}\hat F(x,u)\,dx,
\qquad
\hat F(x,u)=\frac12 F_u(x,u)\cdot u-F(x,u),
\]
and $(F_4)$, we have $\hat F\ge 0$ and hence
\[
0\le \int_{\mathcal G}\hat F(x,u_m)\,dx
=\Phi(u_m)-\frac12\langle \Phi'(u_m),u_m\rangle\to c,
\]
so $c\ge 0$.

If $u_m\to 0$ in $Y$, then by Lemma~\ref{lem-3.4} and the compact embedding
$Y\hookrightarrow L^q(\mathcal K,\mathbb C^2)$ for every $q\in[2,\infty)$, we have $u_m\to 0$ in
$L^q(\mathcal K,\mathbb C^2)$. Using $(F_2)$--$(F_3)$ and the corresponding growth estimate for $F$,
it follows that $\Psi(u_m)=\int_{\mathcal G}F(x,u_m)\,dx\to 0$, and therefore $\Phi(u_m)\to 0$,
so $c=0$. This is alternative (i).

From now on assume $u_m\nrightarrow 0$ in $Y$. Then $c>0$. We claim that
\[
\lambda=\lim_{m\to\infty}|u_m|_2^2>0,
\]
up to a subsequence. Indeed, by $(F_2)$--$(F_3)$ and continuity of $F_u$, there exists $C_0>0$ such that
$|F_u(x,u)|\le C_0|u|$ for all $(x,u)\in\mathcal G\times\mathbb C^2$.
Testing $\Phi'(u_m)$ on $u_m^+-u_m^-$ yields
\[
\langle \Phi'(u_m),u_m^+-u_m^-\rangle
=\|u_m\|^2+\omega\big(|u_m^+|_2^2-|u_m^-|_2^2\big)
-\int_{\mathcal G}F_u(x,u_m)\cdot (u_m^+-u_m^-)\,dx.
\]
Since $|\,|u_m^+|_2^2-|u_m^-|_2^2|\le |u_m|_2^2$ and $|u_m^+-u_m^-|_2=|u_m|_2$, we get
\[
\left|\int_{\mathcal G}F_u(x,u_m)\cdot (u_m^+-u_m^-)\,dx\right|
\le C_0\int_{\mathcal G}|u_m|\,|u_m^+-u_m^-|\,dx
\le C_0|u_m|_2^2.
\]
Moreover $\|u_m^+-u_m^-\|\le \|u_m\|$, hence
$|\langle \Phi'(u_m),u_m^+-u_m^-\rangle|\le \|\Phi'(u_m)\|_{Y^*}\|u_m\|=o(1)$.
Therefore,
\[
\|u_m\|^2\le o(1)+( |\omega|+C_0)|u_m|_2^2.
\]
If $|u_m|_2\to 0$, then $\|u_m\|\to 0$, which implies $\Phi(u_m)\to 0$, contradicting $c>0$.
Hence $\lambda>0$.

Define the normalized sequence $\Psi_m=u_m/|u_m|_2$. Then $(\Psi_m)$ is bounded in $Y$ and
$|\Psi_m|_2=1$. We apply Lemma~\ref{lem:2.2} to $(\Psi_m)$.
The vanishing alternative (i) in Lemma~\ref{lem:2.2} is impossible: if it held, then
$\Psi_m\to 0$ in $L^p(\mathcal G)$ for every $p\in(2,\infty)$, hence $u_m=|u_m|_2\Psi_m\to 0$ in
$L^p(\mathcal K)$ for all $p>2$. Using again the growth
assumptions, we would get $\Psi(u_m)\to 0$, and testing $\Phi'(u_m)$ on $u_m^+-u_m^-$ would force
$\|u_m\|\to 0$, contradicting $c>0$. Therefore Lemma~\ref{lem:2.2}(ii) holds.

Hence there exist $(a_m)\subset\mathbb Z^d$ and $\Psi\in Y\setminus\{0\}$ such that, setting
$\Psi_m:=a_m*\Psi_m$, we have
\[
\Psi_m\rightharpoonup \Psi\ \text{in }Y,
\qquad
\Psi_m\to \Psi\ \text{in }L^p_{\mathrm{loc}}(\mathcal G)\ \text{for all }p\in[2,\infty).
\]
Define $v_m=a_m*u_m=|u_m|_2\,\Psi_m$. Since $|u_m|_2^2\to\lambda>0$, we obtain
\[
v_m\rightharpoonup v=\sqrt{\lambda}\,\Psi\ \text{in }Y,
\qquad
v_m\to v\ \text{in }L^p_{\mathrm{loc}}(\mathcal G)\ \text{for all }p\in[2,\infty),
\]
and $v\not\equiv 0$.

By $\mathbb Z^d$--invariance of $\Phi$ and $\Phi'$, $(v_m)$ is still a $(C)_c$--sequence and
$\Phi'(v_m)\to 0$ in $Y^*$. Since $\Phi':Y\to Y^*$ is weakly sequentially continuous (Lemma~\ref{lem-3.5}),
from $v_m\rightharpoonup v$ we get $\Phi'(v_m)\rightharpoonup \Phi'(v)$ in $Y^*$, hence $\Phi'(v)=0$.
Therefore $v\in K\setminus\{0\}$ and $\Phi(v)>0$. By \eqref{eq-3.22},
\[
\theta<\Phi(v)\le \vartheta.
\]
Since $\mathcal F$ contains representatives of the $\mathbb Z^d$--orbits of $K$, there exist $\bar u_1\in\mathcal F$
and $b\in\mathbb Z^d$ such that $v=b*\bar u_1$. Replacing $a_m$ by $a_m+b$, we may assume $v=\bar u_1$.

Set $w_m=v_m-\bar u_1$. Using Lemma~\ref{lemma2.1} and the local strong
convergence $v_m\to \bar u_1$ in $L^p(\mathcal K)$ for all $p\in[2,\infty)$, we obtain
\begin{equation}\label{eq:BL-splitting-new}
\Phi(v_m)=\Phi(\bar u_1)+\Phi(w_m)+o(1),
\qquad
\|\Phi'(v_m)-\Phi'(w_m)\|_{Y^*}=o(1).
\end{equation}
Since $\Phi'(\bar u_1)=0$ and $\Phi(v_m)\to c$, it follows that $(w_m)$ is a $(C)_{c-\Phi(\bar u_1)}$--sequence.

If $c=\Phi(\bar u_1)$, then $\Phi(w_m)\to 0$ and $\Phi'(w_m)\to 0$ in $Y^*$.
If $w_m\nrightarrow 0$ in $Y$, we may apply the above extraction argument to $(w_m)$ and obtain a nontrivial critical point $ u\in K\setminus\{0\}$ with $\Phi( u)\ge \theta$.
Using again the splitting \eqref{eq:BL-splitting-new} along a translated subsequence would give
$\liminf_m\Phi(w_m)\ge \Phi( u)\ge \theta$, contradicting $\Phi(w_m)\to 0$.
Hence $w_m\to 0$ in $Y$, and therefore
\[
\|u_m-(-a_m)*\bar u_1\|=\|v_m-\bar u_1\|=\|w_m\|\to 0,
\]
which is alternative (ii) with $\ell=1$ and $c=\Phi(\bar u_1)\ge \theta$.

Finally, assume $c>\Phi(\bar u_1)$. Then $c_1=c-\Phi(\bar u_1)>0$ and $(w_m)$ is a $(C)_{c_1}$--sequence which does
not converge to $0$ in $Y$. Repeating the above extraction on $(w_m)$, we obtain a second profile
$\bar u_2\in\mathcal F$, translations $(a_m^2)\subset\mathbb Z^d$ and a new remainder $(w_m^{(2)})$ such that
\[
\Big\|w_m-(a_m^2*\bar u_2)-w_m^{(2)}\Big\|\to 0,
\qquad
(w_m^{(2)})\ \text{is a }(C)_{c-\Phi(\bar u_1)-\Phi(\bar u_2)}\text{--sequence}.
\]
Moreover, choosing the translations so that $|a_m^1-a_m^2|\to\infty$ (where $a_m^1:=-a_m$),
the profiles concentrate in asymptotically disjoint cells and the splitting \eqref{eq:BL-splitting-new}
extends to the sum of the two translated profiles plus the new remainder.

Iterating this procedure, at step $k$ we either obtain that the remainder converges to $0$ in $Y$, or we extract
another profile $\bar u_k\in\mathcal F$ with $\Phi(\bar u_k)\ge \theta$ and a translation sequence $(a_m^k)\subset\mathbb Z^d$.
Since each extracted profile carries energy at least $\theta$, the iteration must stop after at most $[c/\theta]$ steps.
Therefore there exist $\ell\le [c/\theta]$, $\bar u_1,\dots,\bar u_\ell\in\mathcal F$ and sequences
$(a_m^i)\subset\mathbb Z^d$ such that
\[
\Big\|u_m-\sum_{i=1}^\ell (a_m^i*\bar u_i)\Big\|\to 0,
\qquad
\sum_{i=1}^\ell \Phi(\bar u_i)=c.
\]
This is alternative (ii) and the proof is complete.
\end{proof}

	\par
	For $\ell \in \mathbb{N}$ and a finite set $\mathcal{B} \subset Y$ we define
	\[
	[\mathcal{B}, \ell]
	=\left\{\sum_{i=1}^{j}\big(a_{i} * u_{i}\big):
	1 \leq j \leq \ell,\ a_{i} \in \mathbb{Z}^{d},\ u_{i} \in \mathcal{B}\right\}.
	\]
	An argument similar to that in \cite{MR1070929} shows that
	\begin{equation}\label{eq-3.25}
		\inf \left\{\big\|u-u^{\prime}\big\|: u, u^{\prime} \in[\mathcal{B}, \ell],\ u \neq u^{\prime}\right\}>0.
	\end{equation}
	As a consequence of Lemma \ref{lem-3.10} we have the following:

    \begin{Lem}\label{lem-3.11}
Assume \eqref{eq-3.21} holds. Then $\Phi$ satisfies $(\Phi_{5})$.
\end{Lem}

\begin{proof}
Let $I\subset(0,\infty)$ be a compact interval and set $d=\max I$.

If $d<\theta$, then for any $(C)_c$-sequence with $c\in I$ Lemma~\ref{lem-3.10} forces $c=0$,
which is impossible since $I\subset(0,\infty)$. Hence there are no $(C)_c$-sequences with $c\in I$,
and $(\Phi_5)$ holds trivially.

From now on assume $d\ge \theta$ and set
\[
\ell=[d/\theta]\ge 1,
\qquad
\mathcal A=[\mathcal F,\ell].
\]
Let $P^{+}:Y\to Y^{+}$ be the orthogonal projection. Since the $\mathbb Z^{d}$-action is isometric on $Y$
and commutes with the splitting $Y=Y^{-}\oplus Y^{+}$, we have
\[
P^{+}[\mathcal F,\ell]=[P^{+}\mathcal F,\ell].
\]
Applying \eqref{eq-3.25} with $\mathcal B=P^{+}\mathcal F$ yields
\[
\inf\Big\{\|u_{1}^{+}-u_{2}^{+}\|:\ u_{1},u_{2}\in\mathcal A,\ u_{1}^{+}\ne u_{2}^{+}\Big\}>0.
\]

Now let $(u_m)$ be a $(C)_c$-sequence with $c\in I$. By Lemma~\ref{lem-3.10} there exist
$\ell_{0}\le [c/\theta]\le \ell$, profiles $\bar u_{i}\in\mathcal F$ and translations $a_m^{i}\in\mathbb Z^{d}$
such that
\[
\Big\|u_m-\sum_{i=1}^{\ell_{0}}(a_m^{i}*\bar u_{i})\Big\|\to 0 \quad\text{in }Y.
\]
Since $\sum_{i=1}^{\ell_{0}}(a_m^{i}*\bar u_{i})\in[\mathcal F,\ell]\;=\;\mathcal A$, it follows that
$\mathrm{dist}(u_m,\mathcal A)\to 0$. Hence $\mathcal A$ is a $(C)_I$-attractor.

Finally, for every $u\in\mathcal A$ we have
\[
\|u\|\le \ell\,\max\{\|\bar u\|:\bar u\in\mathcal F\},
\]
so $\mathcal A$ is bounded in $Y$. Therefore $\Phi$ satisfies $(\Phi_{5})$.
\end{proof}

	\noindent\textbf{Proof of Theorem \ref{theo-1.1}.}
	We first prove the existence statement. Recall that
	\[
	\Phi(u)
	= \frac{1}{2}\big(\|u^{+}\|^{2}-\|u^{-}\|^{2}+\omega|u|_{2}^{2}\big)
	- \Psi(u),
	\qquad
	\Psi(u) = \int_{\mathcal{K}}F(x,u)\,dx,
	\]
	defined on \(Y=Y^{-}\oplus Y^{+}=\mathcal{D}(|A|^{1/2})\), where the splitting is given by the spectral
	decomposition of \(A\).
	
	Set
	\[
	M=Y^{-},\qquad N=Y^{+}.
	\]
	By Lemma~\ref{lem-3.5} and Theorem~\ref{theo-2.4}, the functional \(\Phi\) satisfies
	\((\Phi_{0})\) and \((\Phi_{1})\) with respect to the splitting \(Y=M\oplus N\).
	Lemma~\ref{lem-3.6} yields \((\Phi_{2})\), and Lemma~\ref{lem-3.8} gives the geometric condition \((\Phi_{3})\).
	Hence all assumptions of the abstract linking theorem, Theorem~\ref{theo-2.2}, are fulfilled.
	
	Therefore there exists a level \(c\ge\eta>0\) and a sequence \((u_m)\subset Y\) such that
	\[
	\Phi(u_m)\to c,
	\qquad
	(1+\|u_m\|)\,\|\Phi'(u_m)\|_{Y^*}\to 0,
	\]
	that is, \((u_m)\) is a \((C)_c\)-sequence at level \(c\ge\eta\).
	By Lemma~\ref{lem-3.9} the sequence $(u_m)$ is bounded in $Y$.
Up to a subsequence we may assume
\[
u_m\rightharpoonup \tilde u \quad\text{in }Y.
\]

We claim that $(u_m)$ does not vanish in the sense of Lemma~\ref{lem:2.2}.
Assume by contradiction that the vanishing alternative holds. Then
\[
u_m\to 0 \quad\text{in }L^{q}(\mathcal G,\mathbb C^{2}) \ \text{for every }q\in(2,\infty).
\]
Using $(F_{0})$--$(F_{2})$ and the estimate $|F_u(x,u)|\le C|u|$ from $(F_{3})$, we obtain
\[
\int_{\mathcal G}F(x,u_m)\,dx \to 0,
\qquad
\int_{\mathcal G}\hat F(x,u_m)\,dx \to 0,
\]
where
\[
\hat F(x,u)=\frac12\,F_u(x,u)\cdot u - F(x,u).
\]
Since
\[
\Phi(u)-\frac12\langle \Phi'(u),u\rangle=\int_{\mathcal G}\hat F(x,u)\,dx,
\]
and $(1+\|u_m\|)\,\|\Phi'(u_m)\|_{Y^*}\to 0$, we have
\[
c=\lim_{m\to\infty}\Phi(u_m)
=\lim_{m\to\infty}\left(\int_{\mathcal G}\hat F(x,u_m)\,dx
+\frac12\langle \Phi'(u_m),u_m\rangle\right)=0,
\]
which contradicts $c\ge \eta>0$. Therefore vanishing is impossible.

Hence we are in the compactness modulo translations alternative of Lemma~\ref{lem:2.2}.
There exist a sequence $(a_m)\subset\mathbb Z^{d}$ and $u\in Y\setminus\{0\}$ such that, setting
\[
v_m=a_m*u_m,
\]
we have
\[
v_m\rightharpoonup u \quad\text{in }Y,
\qquad
v_m\to u \quad\text{in }L^{q}(\mathcal K,\mathbb C^{2})\ \text{for all }q\in[2,\infty),
\]
and almost everywhere on $\mathcal K$.

By $\mathbb Z^{d}$-invariance of $\Phi$ and $\Phi'$, the sequence $(v_m)$ is still a $(C)_c$-sequence and
\[
\Phi(v_m)=\Phi(u_m)\to c,
\qquad
(1+\|v_m\|)\,\|\Phi'(v_m)\|_{Y^*}\to 0.
\]
By Lemma~\ref{lem-3.5}, $\Phi'$ is weakly sequentially continuous on $Y$. Since
$\Phi'(v_m)\to 0$ in $Y^*$ and $v_m\rightharpoonup u$ in $Y$, we obtain
\[
\Phi'(u)=0.
\]
Thus $u\ne 0$ is a critical point of $\Phi$. By Proposition~\ref{pro-3.1}, every nontrivial critical point of \(\Phi\) corresponds to a bound
	state of frequency \(\omega\) for the NLDE \eqref{eq-1.5}. This proves the existence of at least one
	bound state.
	
	We now prove the multiplicity statement. Assume by contradiction that the NLDE \eqref{eq-1.5} has only finitely many
	geometrically distinct bound states. In other words, \eqref{eq-3.21} holds, i.e.,
	$K/\mathbb{Z}^{d}$ is a finite set.
	
	By Lemma~\ref{lem-3.5}, Lemma~\ref{lem-3.6}, Lemma~\ref{lem-3.7} and
	Lemma~\ref{lem-3.11}, the functional $\Phi$ satisfies
	\[
	(\Phi_{0})\text{--}(\Phi_{5}).
	\]
	Therefore, by Theorem~\ref{theo-2.3}, $\Phi$ possesses an unbounded sequence
	of critical values
	\[
	0 < c_{1} \le c_{2} \le \cdots,\qquad c_{n}\to+\infty,
	\]
	which contradicts the finiteness of $K/\mathbb{Z}^{d}$ in \eqref{eq-3.21}.
	Hence NLDE \eqref{eq-1.5} must have infinitely many geometrically distinct bound
	states, completing the proof of Theorem~\ref{theo-1.1}.

	\subsection{Proof of Theorem \ref{theo-1.2}}
	
	We consider as before the functionals
	\[
	\Psi(u)=\int_{\mathcal{G}} F(x,u)\,dx,
	\qquad
	\Phi(u)
	=\frac{1}{2}\big(\|u^{+}\|^{2}-\|u^{-}\|^{2}+\omega|u|_{2}^{2}\big)-\Psi(u),
	\]
	defined on \(Y\), see \eqref{eq-3.6}.Choose \(\gamma>\gamma_{0}=a+|\omega|+\sup_{\mathcal{G}}V\), and set
	\[
	Y_{0}=\big(E_{\gamma}-E_{\gamma_0}\big)L^{2}(\mathcal{G},\mathbb{C}^{2}).
	\]
	Then \((Y_{n})_{n\in\mathbb{N}}\) is an increasing sequence of finite-dimensional subspaces of \(Y^{+}\), and
	\begin{equation}\label{eq-3.27}
		\gamma_{0}|u|_{2}^{2}\le\|u\|^{2}\le\gamma|u|_{2}^{2}
		\quad\text{for all }u\in Y_{0},
	\end{equation}
	by the spectral calculus for \(|A|\).
	
	\begin{Lem}\label{lem-3.12}
Assume $(\omega)$, $(F_{0})$--$(F_{2})$ and $(F_{6})$--$(F_{7})$. Then the following hold:
\begin{itemize}
\item[$(a)$] $\Psi$ is weakly sequentially lower semicontinuous on $Y$ and $\Phi':Y\to Y^*$ is weakly sequentially continuous.
Moreover, for every $c>0$ there exists $\zeta>0$ such that
\[
\|u\|<\zeta\|u^{+}\|\quad\text{for all }u\in\Phi_{c}=\{u\in Y:\Phi(u)\ge c\}.
\]
\item[$(b)$] There exists $\rho>0$ such that
\[
\eta=\inf\{\Phi(u):u\in Y^{+},\ \|u\|=\rho\}>0.
\]
\item[$(c)$] For every $n\in\mathbb N$ one has $\sup\Phi(E_{n})<\infty$, and there exists $R_{n}>0$ such that
\[
\sup\Phi(E_{n}\setminus B_{n})\le \inf\Phi(B_{\rho}),
\]
where $E_n=Y^-\oplus Y_n$, $B_{n}=\{u\in E_{n}:\|u\|\le R_{n}\}$ and $B_{\rho}=\{u\in Y^{+}:\|u\|\le\rho\}$.
\end{itemize}
\end{Lem}

\begin{proof}
$(a)$ Let $u_j\rightharpoonup u$ in $Y$. Fix a finite set $\Lambda\subset\mathbb Z^d$ and set
\[
\mathcal K_\Lambda=\bigcup_{k\in\Lambda}T^k(\mathcal K).
\]
Since $\mathcal K_\Lambda$ is compact and the $Y$-norm is equivalent to the $H^{1/2}$-norm (Lemma~\ref{lem-3.4}),
the embedding $Y\hookrightarrow L^{p}(\mathcal K_\Lambda,\mathbb C^{2})$ is compact for every $p\in[2,\infty)$.
Hence $u_j\to u$ in $L^{p}(\mathcal K_\Lambda)$ for every $p\in[2,\infty)$ and a.e. on $\mathcal K_\Lambda$.
Using $(F_0)$--$(F_2)$ and the assumed upper growth, we obtain
\[
\int_{\mathcal K_\Lambda}F(x,u_j)\,dx\to \int_{\mathcal K_\Lambda}F(x,u)\,dx.
\]
Since $F\ge 0$, we have for every $\Lambda$
\[
\int_{\mathcal K_\Lambda}F(x,u)\,dx
=\lim_{j\to\infty}\int_{\mathcal K_\Lambda}F(x,u_j)\,dx
\le \liminf_{j\to\infty}\int_{\mathcal G}F(x,u_j)\,dx,
\]
and taking the supremum over finite $\Lambda$ yields
\[
\Psi(u)=\int_{\mathcal G}F(x,u)\,dx
\le \liminf_{j\to\infty}\int_{\mathcal G}F(x,u_j)\,dx
=\liminf_{j\to\infty}\Psi(u_j),
\]
that is, $\Psi$ is weakly sequentially lower semicontinuous.

To show weak sequential continuity of $\Phi'$, let $v\in Y$ be fixed. The linear part is weakly continuous, so it suffices to prove
\[
\int_{\mathcal G}F_u(x,u_j)\cdot v\,dx \to \int_{\mathcal G}F_u(x,u)\cdot v\,dx.
\]
Split $\mathcal G=\mathcal K_\Lambda\cup(\mathcal G\setminus\mathcal K_\Lambda)$. On $\mathcal K_\Lambda$ we have
$u_j\to u$ in $L^p(\mathcal K_\Lambda)$, hence $F_u(x,u_j)\to F_u(x,u)$ in $L^{p'}(\mathcal K_\Lambda)$ for a suitable pair $(p,p')$,
and therefore the integral over $\mathcal K_\Lambda$ converges.
On the complement we use Hölder:
\[
\left|\int_{\mathcal G\setminus\mathcal K_\Lambda}\big(F_u(x,u_j)-F_u(x,u)\big)\cdot v\,dx\right|
\le \|F_u(\cdot,u_j)-F_u(\cdot,u)\|_{L^{p'}(\mathcal G)}\,\|v\|_{L^{p}(\mathcal G\setminus\mathcal K_\Lambda)}.
\]
The first factor is bounded uniformly in $j$ by the assumed growth and boundedness of $(u_j)$ in $Y$,
while the second factor can be made arbitrarily small by choosing $\Lambda$ large, since $v\in L^p(\mathcal G)$.
Thus the whole integral converges, proving weak sequential continuity of $\Phi'$.

Finally, the estimate $\|u\|<\zeta\|u^+\|$ on the superlevel set $\Phi_c$ follows exactly as in Lemma~\ref{lem-3.5},
using only $(\omega)$, Lemma~\ref{lem-3.4} and $F\ge 0$.

$(b)$ By $(F_2)$ and the assumed upper growth, there exist $p>2$ and, for every $\varepsilon>0$, a constant $C_\varepsilon>0$ such that
\[
F(x,u)\le \varepsilon |u|^2 + C_\varepsilon |u|^p
\quad\text{for all }(x,u)\in\mathcal G\times\mathbb C^2.
\]
Hence, for all $u\in Y$,
\[
\Psi(u)\le \varepsilon |u|_2^2 + C_\varepsilon |u|_p^p
\le \frac{\varepsilon}{a}\|u\|^2 + C_\varepsilon C_p^p \|u\|^p,
\]
where we used $a|u|_2^2\le \|u\|^2$ and $|u|_p\le C_p\|u\|$.
If $u\in Y^+$, then by \eqref{eq-3.7} we get
\[
\Phi(u)\ge \frac{a-|\omega|}{2a}\|u\|^2-\frac{\varepsilon}{a}\|u\|^2-C_\varepsilon C_p^p\|u\|^p.
\]
Choose $\varepsilon>0$ so small that $\delta=\frac{a-|\omega|}{2a}-\frac{\varepsilon}{a}>0$, and then choose $\rho>0$ such that
$C_\varepsilon C_p^p\rho^{p-2}\le \frac{\delta}{2}$. Then for $\|u\|=\rho$,
\[
\Phi(u)\ge \frac{\delta}{2}\rho^2>0,
\]
hence $\eta>0$.

$(c)$ Fix $n\in\mathbb N$ and let $u=u^-+u^+\in E_n=Y^-\oplus Y_n$. Since $F\ge 0$,
\[
\Phi(u)\le \frac12\|u^+\|^2-\frac12\|u^-\|^2+\frac{|\omega|}{2}|u|_2^2
\le \frac12\|u^+\|^2-\frac12\|u^-\|^2+\frac{|\omega|}{2a}\|u\|^2.
\]
In particular, if $\|u^-\|\to\infty$ along a sequence in $E_n$, then $\Phi(u)\to-\infty$.

It remains to consider sequences with $\|u\|\to\infty$ and $\|u^-\|$ bounded. Then necessarily $\|u^+\|\to\infty$ with $u^+\in Y_n$.
Since $Y_n$ is finite dimensional, all norms on $Y_n$ are equivalent; in particular, there exists $C_n>0$ such that
\[
|u^+|_\infty \le C_n\|u^+\|,\qquad |u^+|_2 \ge \frac{1}{C_n}\|u^+\|.
\]
Let $w=u/\|u^+\|$, so $w^+\in Y_n$ and $\|w^+\|=1$. Up to a subsequence, $w^+\to w_0^+$ in $Y_n$, hence in $L^2_{\mathrm{loc}}$,
and $w_0^+\ne 0$. Choose a finite $\Lambda$ such that
\[
\int_{\mathcal K_\Lambda}|w_0^+|^2\,dx>0.
\]
Then for $j$ large,
\[
\int_{\mathcal K_\Lambda}|w_j|^2\,dx \ge \frac12\int_{\mathcal K_\Lambda}|w_0^+|^2\,dx = c_\Lambda>0.
\]
By $(F_6)$, for every $M>0$ there exists $R(M)>0$ such that
\[
F(x,z)\ge M|z|^2 \quad\text{for all }x\in\mathcal G,\ |z|\ge R(M).
\]
Since $u_j=\|u_j^+\|\,w_j$ and $\|u_j^+\|\to\infty$, we have $|u_j(x)|\to\infty$ for a.e. $x$ with $w_0^+(x)\ne 0$,
hence for $j$ large $|u_j(x)|\ge R(M)$ on a subset of $\mathcal K_\Lambda$ of positive measure.
Using Fatou’s lemma and the lower bound above we obtain
\[
\liminf_{j\to\infty}\frac{1}{\|u_j^+\|^2}\int_{\mathcal K_\Lambda}F(x,u_j)\,dx
\ge M\liminf_{j\to\infty}\int_{\mathcal K_\Lambda}|w_j|^2\,dx
\ge M c_\Lambda.
\]
Since $F\ge 0$,
\[
\frac{\Phi(u_j)}{\|u_j^+\|^2}
\le \frac12\frac{\|u_j^+\|^2}{\|u_j^+\|^2}
+\frac{|\omega|}{2a}\frac{\|u_j\|^2}{\|u_j^+\|^2}
-\frac{1}{\|u_j^+\|^2}\int_{\mathcal K_\Lambda}F(x,u_j)\,dx.
\]
The middle ratio is bounded, while the last term is at least $M c_\Lambda$ in the limit inferior. Choosing $M$ so large that
$M c_\Lambda$ dominates the bounded positive contributions, we get $\Phi(u_j)\to-\infty$, a contradiction to boundedness from below.
Therefore, $\Phi(u)\to-\infty$ as $\|u\|\to\infty$ with $u\in E_n$.

Consequently $\sup\Phi(E_n)<\infty$, and since $\Phi(u)\to-\infty$ along $\|u\|\to\infty$ in $E_n$, we can choose $R_n>0$ such that
\[
\sup\{\Phi(u):u\in E_n,\ \|u\|\ge R_n\}\le \inf\Phi(B_\rho).
\]
This is exactly $\sup\Phi(E_n\setminus B_n)\le \inf\Phi(B_\rho)$.
\end{proof}
	
	As a consequence of Lemma \ref{lem-3.12} we have the following results.
	
	\begin{Lem}\label{lem-3.13}
		Under the assumptions of Lemma \ref{lem-3.12}, one has
		\[
		\Phi(u)\le 0 \quad\text{for all } u\in\partial Q,
		\]
		where
		\[
		Q= \big\{u = u^{-} + s e_{1} : u^{-}\in Y^{-},\ s\ge 0,\ \|u\|\le R_{1}\big\},
		\]
		and \(R_{1}>0\) is given by Lemma \ref{lem-3.12}(c) for \(n=1\).
	\end{Lem}
	
	\begin{proof}
By definition,
\[
Q=\{u=u^-+se_1:\ u^-\in Y^-,\ s\ge 0,\ \|u\|\le R_1\}.
\]
Hence
\[
\partial Q=(B_{R_1}\cap Y^-)\ \cup\ \{u=u^-+se_1:\ u^-\in Y^-,\ s\ge 0,\ \|u\|=R_1\},
\]
and the second set is contained in $\{u\in E_1:\ \|u\|=R_1\}$ since $E_1=Y^-\oplus Y_1$ and $Y_1=\mathrm{span}\{e_1\}$.

First, if $u\in Y^-$, then $u^+=0$ and, since $F\ge 0$,
\[
\Phi(u)
=-\frac12\|u\|^2+\frac{\omega}{2}|u|_2^2-\Psi(u)
\le -\frac12\big(\|u\|^2-\omega|u|_2^2\big)
\le -\frac{a-|\omega|}{2a}\,\|u\|^2\le 0,
\]
where we used \eqref{eq-3.8} and Lemma~\ref{lem-3.4}.

Next, consider $u\in E_1$ with $\|u\|=R_1$. Set $B_1=\{v\in E_1:\ \|v\|< R_1\}$, so $u\in E_1\setminus B_1$.
From the proof of Lemma~\ref{lem-3.12}(b), there exist $\delta>0$, $p>2$ and a constant $C>0$ such that for all $v\in Y^+$,
\[
\Phi(v)\ge \delta\|v\|^2-C\|v\|^p.
\]
With the choice of $\rho>0$ made there, one has $C\rho^{p-2}\le \delta/2$, and therefore for every $v\in Y^+$ with $\|v\|\le \rho$,
\[
\Phi(v)\ge \|v\|^2\big(\delta-C\|v\|^{p-2}\big)\ge \frac{\delta}{2}\|v\|^2\ge 0.
\]
Hence $\inf\Phi(B_\rho)=0$, where $B_\rho=\{v\in Y^+:\ \|v\|\le \rho\}$.

Now Lemma~\ref{lem-3.12}(c) with $n=1$ yields
\[
\sup\Phi(E_1\setminus B_1)\le \inf\Phi(B_\rho)=0,
\]
so in particular $\Phi(u)\le 0$ for every $u\in E_1$ with $\|u\|=R_1$.

Combining the two cases, we conclude that $\Phi(u)\le 0$ for all $u\in\partial Q$.
\end{proof}

	\begin{Lem}\label{lem-3.14}
		Under the assumptions of Lemma \ref{lem-3.12}, every \((C)_{c}\)-sequence is bounded.
	\end{Lem}
\begin{proof}
Let $(u_m)\subset Y$ be a $(C)_c$-sequence, namely
\[
\Phi(u_m)\to c,
\qquad
(1+\|u_m\|)\,\|\Phi'(u_m)\|_{Y^*}\to 0 .
\]
Recall that
\[
\Phi(u)-\frac12\langle \Phi'(u),u\rangle
=\int_{\mathcal K}\hat F(x,u)\,dx,
\qquad
\hat F(x,u)=\frac12\,F_u(x,u)\cdot u-F(x,u).
\]
By $(F_7)$ we have $\hat F(x,u)\ge 0$ for all $(x,u)$, hence there exists $C_1>0$ such that for $m$ large,
\begin{equation}\label{eq-3.28-new}
0\le \int_{\mathcal K}\hat F(x,u_m)\,dx
=\Phi(u_m)-\frac12\langle \Phi'(u_m),u_m\rangle
\le C_1 .
\end{equation}

Assume by contradiction that $\|u_m\|\to\infty$.
Introduce the equivalent norm
\[
\|u\|_\omega^2
=\|u\|^2+\omega\big(|u^+|_2^2-|u^-|_2^2\big).
\]
Set
\[
\omega_0=a-|\omega|>0.
\]
Using \eqref{eq-3.7}--\eqref{eq-3.8} and Lemma~\ref{lem-3.4}, we obtain
\begin{equation}\label{eq-3.31-new}
\omega_0|u|_2^2\le \|u\|_\omega^2,
\qquad
\frac{\omega_0}{a}\,\|u\|^2\le \|u\|_\omega^2\le \frac{a+|\omega|}{a}\,\|u\|^2
\quad\text{for all }u\in Y.
\end{equation}
In particular, $\|u_m\|_\omega\to\infty$.

Set
\[
v_m=\frac{u_m}{\|u_m\|_\omega}.
\]
Then $\|v_m\|_\omega=1$ and $(v_m)$ is bounded in $Y$. By Sobolev embeddings on $\mathcal K$,
\[
|v_m|_s\le C_s\quad\text{for all }s\in[2,\infty),
\]
with constants $C_s$ independent of $m$.

\textbf{Step 1. }
Let $r>0$ and $c_2>0$ be given by $(F_7)(i)$, so that $\hat F(x,u)\ge c_2|u|^2$ for $|u|\ge r$.
We claim that for every $\rho>0$ there exists $a_\rho>0$ such that
\begin{equation}\label{eq-3.29-new}
\hat F(x,u)\ge a_\rho |u|^2
\quad\text{for all }x\in\mathcal G,\ |u|\ge \rho.
\end{equation}
If $\rho\ge r$ we take $a_\rho=c_2$. If $0<\rho<r$, using $\mathbb Z^d$-periodicity and cocompactness we reduce $x$ to
$\mathcal K$ and consider the compact set
\[
\mathcal C_\rho=\{(x,u)\in \mathcal K\times\mathbb C^2:\ \rho\le |u|\le r\}.
\]
Since $\hat F(x,u)>0$ for $u\ne 0$ by $(F_7)$ and $\hat F$ is continuous, we have
\[
m_\rho=\min_{\mathcal C_\rho}\hat F>0.
\]
Thus for $\rho\le |u|\le r$ we have $\hat F(x,u)\ge m_\rho\ge \frac{m_\rho}{r^2}|u|^2$, while for $|u|\ge r$ we have
$\hat F(x,u)\ge c_2|u|^2$. Hence \eqref{eq-3.29-new} holds with
\[
a_\rho=\min\left\{c_2,\frac{m_\rho}{r^2}\right\}>0.
\]

\textbf{Step 2.}
For $\rho>0$ define
\[
Q_m(\rho)=\{x\in\mathcal K:\ |u_m(x)|\ge \rho\}.
\]
By \eqref{eq-3.28-new} and \eqref{eq-3.29-new},
\[
a_\rho\int_{Q_m(\rho)}|u_m|^2\,dx
\le \int_{Q_m(\rho)}\hat F(x,u_m)\,dx
\le \int_{\mathcal K}\hat F(x,u_m)\,dx
\le C_1,
\]
so
\begin{equation}\label{eq-3.30-new}
\int_{Q_m(\rho)}|u_m|^2\,dx\le \frac{C_1}{a_\rho}.
\end{equation}
Consequently,
\[
\int_{Q_m(\rho)}|v_m|^2\,dx
=\frac{1}{\|u_m\|_\omega^2}\int_{Q_m(\rho)}|u_m|^2\,dx
\le \frac{C_1}{a_\rho\,\|u_m\|_\omega^2}\to 0.
\]
Fix $s\in(2,\infty)$ and choose $\bar s\in(s,\infty)$. By Hölder interpolation we obtain
\begin{equation}\label{eq-3.32-new}
\int_{Q_m(\rho)}|v_m|^s\,dx\to 0
\quad\text{for every }s\in(2,\infty).
\end{equation}

\textbf{Step 3.}
From the expression of $\Phi'$,
\[
\langle \Phi'(u_m),u_m^+-u_m^-\rangle
=\|u_m\|_\omega^2-\int_{\mathcal K}F_u(x,u_m)\cdot (u_m^+-u_m^-)\,dx.
\]
Since $(1+\|u_m\|)\|\Phi'(u_m)\|_{Y^*}\to 0$ and $\|u_m^+-u_m^-\|\le \|u_m\|$, by \eqref{eq-3.31-new} we get
\[
\frac{\langle \Phi'(u_m),u_m^+-u_m^-\rangle}{\|u_m\|_\omega^2}=o(1).
\]
Dividing the identity above by $\|u_m\|_\omega^2$ yields
\[
1-\int_{\mathcal K}\frac{F_u(x,u_m)}{|u_m|}\cdot (v_m^+-v_m^-)\,|v_m|\,dx=o(1),
\]
where we define $\frac{F_u(x,u_m)}{|u_m|}=0$ on $\{u_m=0\}$.
Hence
\begin{equation}\label{eq-int-to-1}
\int_{\mathcal K}\frac{F_u(x,u_m)}{|u_m|}\cdot (v_m^+-v_m^-)\,|v_m|\,dx\to 1.
\end{equation}

Define
\[
J_m=\left\{x\in\mathcal K:\ \frac{|F_u(x,u_m(x))|}{|u_m(x)|}\le \frac{\omega_0}{2}\right\},
\qquad
J_m^c=\mathcal K\setminus J_m.
\]
Using Cauchy--Schwarz and \eqref{eq-3.31-new},
\[
\begin{aligned}
\left|\int_{J_m}\frac{F_u(x,u_m)}{|u_m|}\cdot (v_m^+-v_m^-)\,|v_m|\,dx\right|
&\le \frac{\omega_0}{2}\int_{J_m}|v_m|\,|v_m^+-v_m^-|\,dx \\
&\le \frac{\omega_0}{2}\,|v_m|_2\,|v_m^+-v_m^-|_2
\le \frac{\omega_0}{2}\,|v_m|_2^2
\le \frac12,
\end{aligned}
\]
because $|v_m|_2^2\le \omega_0^{-1}\|v_m\|_\omega^2=\omega_0^{-1}$.
Combining with \eqref{eq-int-to-1} we obtain
\begin{equation}\label{eq-3.34-new}
\liminf_{m\to\infty}
\int_{J_m^c}\frac{F_u(x,u_m)}{|u_m|}\cdot (v_m^+-v_m^-)\,|v_m|\,dx
\ge \frac12.
\end{equation}

\textbf{Step 4.}
By $(F_2)$ there exists $\rho_1>0$ such that
\[
|F_u(x,u)|\le \frac{\omega_0}{2}|u|
\quad\text{whenever }|u|\le \rho_1.
\]
Let $\rho=\max\{\rho_1,r\}$. Then $x\in J_m^c$ implies $|u_m(x)|\ge \rho$ and hence $(F_7)(ii)$ applies:
there exist $\sigma>1$ and $c_3>0$ such that
\begin{equation}\label{eq-3.35-new}
\left(\frac{|F_u(x,u_m)|}{|u_m|}\right)^\sigma
=\frac{|F_u(x,u_m)|^\sigma}{|u_m|^\sigma}
\le c_3\,\hat F(x,u_m)
\quad\text{on }J_m^c.
\end{equation}
Let $\sigma'=\sigma/(\sigma-1)$, so $2\sigma'\in(2,\infty)$.
Using Hölder and then \eqref{eq-3.35-new},
\[
\begin{aligned}
\int_{J_m^c}\frac{|F_u(x,u_m)|}{|u_m|}\,|v_m|\,|v_m^+-v_m^-|\,dx
&\le \left(\int_{J_m^c}\left(\frac{|F_u(x,u_m)|}{|u_m|}\right)^\sigma dx\right)^{1/\sigma}
\left(\int_{J_m^c}\big(|v_m|\,|v_m^+-v_m^-|\big)^{\sigma'}dx\right)^{1/\sigma'} \\
&\le C\left(\int_{J_m^c}\hat F(x,u_m)\,dx\right)^{1/\sigma}
\left(\int_{J_m^c}|v_m|^{2\sigma'}dx\right)^{1/(2\sigma')}
\left(\int_{J_m^c}|v_m^+-v_m^-|^{2\sigma'}dx\right)^{1/(2\sigma')}.
\end{aligned}
\]
The first factor is uniformly bounded by \eqref{eq-3.28-new}.
The third factor is uniformly bounded since $(v_m)$ is bounded in $Y$ and $Y\hookrightarrow L^{2\sigma'}(\mathcal K)$.
Moreover, $J_m^c\subset Q_m(\rho)$, hence by \eqref{eq-3.32-new} with $s=2\sigma'$,
\[
\int_{J_m^c}|v_m|^{2\sigma'}dx\le \int_{Q_m(\rho)}|v_m|^{2\sigma'}dx\to 0.
\]
Therefore the whole right-hand side tends to $0$, which contradicts \eqref{eq-3.34-new}.
This contradiction shows that $\|u_m\|_\omega$ is bounded, and hence $(u_m)$ is bounded in $Y$ by \eqref{eq-3.31-new}.
\end{proof}

	Based on Lemmas \ref{lem-3.10} and \ref{lem-3.11}, we have the following analogue
	for the present hypotheses:

    \begin{Lem}\label{lem-3.15}
Let $(\omega)$, $(F_0)$--$(F_2)$ and $(F_5)$--$(F_7)$ be satisfied.
Assume that $\Phi$ has only finitely many geometrically distinct critical points, namely
$K/\mathbb Z^d$ is finite, where
\[
K=\{u\in Y\setminus\{0\}:\ \Phi'(u)=0\}.
\]
Then for any compact interval $I=[c,d]\subset(0,\infty)$, $\Phi$ has a $(C)_I$-attractor
$\mathcal A\subset Y$ such that $P^+\mathcal A\subset Y^+$ is bounded and
\[
\inf\{\|u^+-v^+\|:\ u,v\in\mathcal A,\ u^+\ne v^+\}>0,
\]
where $P^+:Y\to Y^+$ is the orthogonal projection.
\end{Lem}

\begin{proof}
Assume $K/\mathbb Z^d$ is finite and let $\mathcal F\subset K$ be a finite set of representatives
of the $\mathbb Z^d$-orbits in $K$. Since $\Phi$ is $\mathbb Z^d$-invariant and $\Phi'$ is odd, we may assume
$\mathcal F=-\mathcal F$.

For any $u\in K$ we have $\Phi'(u)=0$, hence
\[
\Phi(u)=\Phi(u)-\frac12\langle \Phi'(u),u\rangle=\int_{\mathcal K}\hat F(x,u)\,dx,
\qquad
\hat F(x,u)=\frac12F_u(x,u)\cdot u-F(x,u).
\]
By $(F_7)$ one has $\hat F(x,u)>0$ for $u\ne 0$, hence $\Phi(u)>0$ for all $u\in K$.
Therefore there exist constants $0<\theta\le \vartheta$ such that
\begin{equation}\label{eq-3.22-super}
0<\theta<\min_{\mathcal F}\Phi=\min_K\Phi\le \max_K\Phi=\max_{\mathcal F}\Phi<\vartheta.
\end{equation}

Fix $I=[c,d]\subset(0,\infty)$ and let $[r]$ denote the integer part of $r$.
Set
\[
\ell=[d/\theta],
\qquad
\mathcal A=[\mathcal F,\ell]
=\left\{\sum_{i=1}^{j}(a_i*u_i):\ 1\le j\le \ell,\ a_i\in\mathbb Z^d,\ u_i\in\mathcal F\right\}.
\]

\smallskip
\noindent\textbf{Claim 1. $\mathcal A$ is a $(C)_I$-attractor for $\Phi$.}
Let $(u_m)\subset Y$ be a $(C)_c$-sequence with $c\in I$. By Lemma~\ref{lem-3.14}, $(u_m)$ is bounded in $Y$.
If $u_m\to 0$ in $Y$, then $\Phi(u_m)\to 0$, contradicting $c>0$. Hence $u_m\not\to 0$.

Apply Lemma~\ref{lem:2.2} to the normalized
sequence $\Psi_m=u_m/|u_m|_2$.
The vanishing alternative of Lemma~\ref{lem:2.2} is excluded as follows: if for some $R>0$,
\[
\sup_{x\in\mathcal G}\int_{B_R(x)}|\Psi_m|^2\to 0,
\]
then $\Psi_m\to 0$ in $L^p(\mathcal K)$ for all $p\in(2,\infty)$; hence $u_m\to 0$ in $L^p(\mathcal K)$.
Using $(F_2)$ and $(F_7)$, this implies $\int_{\mathcal K}F(x,u_m)\,dx\to 0$ and
$\int_{\mathcal K}\hat F(x,u_m)\,dx\to 0$, so
\[
c=\lim_{m\to\infty}\Phi(u_m)=\lim_{m\to\infty}\left(\int_{\mathcal K}\hat F(x,u_m)\,dx+\frac12\langle\Phi'(u_m),u_m\rangle\right)=0,
\]
a contradiction. Therefore the compactness modulo translations case holds.

Thus there exist $(a_m^1)\subset\mathbb Z^d$ and a nontrivial limit $\bar u_1\in Y\setminus\{0\}$ such that
$v_m^1=a_m^1*u_m\rightharpoonup \bar u_1$ in $Y$ and $v_m^1\to \bar u_1$ in $L^q_{\mathrm{loc}}(\mathcal G)$ for all $q\in[2,\infty)$.
By $\mathbb Z^d$-invariance, $(v_m^1)$ is still a $(C)_c$-sequence. Using weak sequential continuity of $\Phi'$ we obtain $\Phi'(\bar u_1)=0$, hence $\bar u_1\in K$.
By definition of $\mathcal F$, after composing with one fixed translation we may assume $\bar u_1\in\mathcal F$.

Set $w_m^1=v_m^1-\bar u_1$. Using $(F_5)$ and Lemma~\ref{lem-3.10},
we have
\[
\Phi(v_m^1)=\Phi(\bar u_1)+\Phi(w_m^1)+o(1),
\qquad
\Phi'(v_m^1)=\Phi'(w_m^1)+o(1),
\]
so $(w_m^1)$ is a $(C)_{c-\Phi(\bar u_1)}$-sequence.
If $c=\Phi(\bar u_1)$ then $w_m^1\to 0$ in $Y$ and we are done with $\ell=1$.
If $c>\Phi(\bar u_1)$, we repeat the same argument with $(w_m^1)$ in place of $(u_m)$.
Each extracted nontrivial profile $\bar u_i$ satisfies $\Phi(\bar u_i)\ge\theta$ by \eqref{eq-3.22-super},
and the translations can be chosen so that $|a_m^i-a_m^j|\to\infty$ for $i\ne j$.
Hence the procedure must terminate after at most $[c/\theta]\le [d/\theta]=\ell$ steps.

Therefore there exist $\ell_0\le \ell$, profiles $\bar u_1,\dots,\bar u_{\ell_0}\in\mathcal F$ and translations
$(a_m^i)\subset\mathbb Z^d$ such that
\[
\left\|u_m-\sum_{i=1}^{\ell_0}(a_m^i*\bar u_i)\right\|\to 0.
\]
This shows $\mathrm{dist}(u_m,\mathcal A)\to 0$, so $\mathcal A$ is a $(C)_I$-attractor.

\smallskip
\noindent\textbf{Claim 2. $P^+\mathcal A$ is bounded and uniformly separated.}
Since the $\mathbb Z^d$-action is isometric and commutes with the splitting, $P^+(a*u)=a*(P^+u)$.
Hence $P^+\mathcal A=[P^+\mathcal F,\ell]$, and for any $u\in\mathcal A$,
\[
\|u^+\|
\le \sum_{i=1}^{\ell}\|\bar u_i^+\|
\le \ell\,\max\{\|\bar u^+\|:\ \bar u\in\mathcal F\},
\]
so $P^+\mathcal A$ is bounded.

Finally, applying \eqref{eq-3.25} to the finite set $\mathcal B=P^+\mathcal F\subset Y^+$ yields
\[
\inf\{\|u^+-v^+\|:\ u,v\in\mathcal A,\ u^+\ne v^+\}
=\inf\{\|w-w'\|:\ w,w'\in[P^+\mathcal F,\ell],\ w\ne w'\}>0.
\]
The proof is complete.
\end{proof}

\noindent\textbf{Proof of Theorem \ref{theo-1.2}.}
Let $M=Y^-$ and $N=Y^+$.
By Lemma~\ref{lem-3.12}(a), $\Psi$ is weakly sequentially lower semicontinuous and $\Phi'$ is weakly sequentially continuous;
moreover, the cone condition on superlevel sets holds, hence $(\Phi_0)$ is satisfied.
Lemma~\ref{lem-3.12}(b) yields $\rho>0$ and $\eta>0$ such that
\[
\Phi(u)\ge \eta \quad\text{for all }u\in Y^+ \text{ with }\|u\|=\rho,
\]
so $(\Phi_2)$ holds.
Lemma~\ref{lem-3.12}(c) and Lemma~\ref{lem-3.13} give the linking geometry, hence $(\Phi_1)$ holds.
Finally, Lemma~\ref{lem-3.14} provides the Cerami compactness condition at all levels, hence $(\Phi_3)$ holds.
Therefore all assumptions of Theorem~\ref{theo-2.2} are fulfilled.

Consequently, there exist $c\ge \eta$ and a $(C)_c$-sequence $(u_m)\subset Y$ such that
\[
\Phi(u_m)\to c,
\qquad
(1+\|u_m\|)\,\|\Phi'(u_m)\|_{Y^*}\to 0.
\]
By Lemma~\ref{lem-3.14}, $(u_m)$ is bounded in $Y$.
Up to a subsequence, $u_m\rightharpoonup u$ in $Y$.

We claim that $u\ne 0$ after a suitable $\mathbb Z^d$-translation.
Assume by contradiction that the vanishing alternative holds in the sense of Lemma~\ref{lem:2.2}: for some $R>0$,
\[
\sup_{x\in\mathcal G}\int_{B_R(x)}|u_m(y)|^2\,dy\to 0.
\]
Then $u_m\to 0$ in $L^p(\mathcal K)$ for every $p\in(2,\infty)$, hence by $(F_2)$ and $(F_7)$,
\[
\int_{\mathcal K}F(x,u_m)\,dx\to 0,
\qquad
\int_{\mathcal K}\hat F(x,u_m)\,dx\to 0.
\]
Using the identity
\[
\Phi(u_m)-\frac12\langle \Phi'(u_m),u_m\rangle=\int_{\mathcal K}\hat F(x,u_m)\,dx,
\]
and $\langle \Phi'(u_m),u_m\rangle\to 0$ from the Cerami condition, we obtain $c=0$, contradicting $c\ge \eta>0$.
Thus vanishing is excluded.

By Lemma~\ref{lem:2.2}, there exist $(a_m)\subset\mathbb Z^d$ and $v\in Y\setminus\{0\}$ such that, setting $v_m=a_m*u_m$,
\[
v_m\rightharpoonup v \text{ in }Y,
\qquad
v_m\to v \text{ in }L^q_{\mathrm{loc}}(\mathcal G,\mathbb C^2)\ \text{for all }q\in[2,\infty).
\]
By $\mathbb Z^d$-invariance, $(v_m)$ is still a $(C)_c$-sequence, hence $\Phi'(v_m)\to 0$ in $Y^*$.
Using the weak sequential continuity of $\Phi'$ (Lemma~\ref{lem-3.12}(a)), we obtain $\Phi'(v)=0$.
Since $v\ne 0$, $v$ is a nontrivial critical point of $\Phi$.
By Proposition~\ref{pro-3.1}, $v$ yields a bound state of NLDE \eqref{eq-1.5}.
This proves the existence part of Theorem~\ref{theo-1.2}.

For the multiplicity statement, assume in addition that $\Phi$ is even, equivalently
\[
F(x,-u)=F(x,u)\quad\text{for all }(x,u)\in\mathcal G\times\mathbb C^2,
\]
and suppose by contradiction that NLDE \eqref{eq-1.5} has only finitely many geometrically distinct bound states.
Then $\Phi$ has only finitely many geometrically distinct nontrivial critical points.
By Lemma~\ref{lem-3.15}, $\Phi$ satisfies condition $(\Phi_5)$.
Together with Lemma~\ref{lem-3.12}, Lemma~\ref{lem-3.13} and Lemma~\ref{lem-3.14},
all assumptions of Theorem~\ref{theo-2.3} are satisfied.
Therefore $\Phi$ admits an unbounded sequence of positive critical values, contradicting the finiteness of
geometrically distinct critical points.
Hence NLDE \eqref{eq-1.5} possesses infinitely many geometrically distinct bound states.

\section*{Acknowledgments}
We would like to thank the anonymous referee for his/her careful readings of our manuscript and the useful comments. 

\medskip
{\bf Funding:} This work is supported by National Natural Science Foundation of China (12301145, 12261107, 12561020) and Yunnan Fundamental Research Projects (202301AU070144, 202401AU070123).

\medskip
{\bf Author Contributions:} All the authors wrote the main manuscript text together and these authors contributed equally to this work.

\medskip
{\bf Data availability:}  Data sharing is not applicable to this article as no new data were created or analyzed in this study.

\medskip
{\bf Conflict of Interests:} The author declares that there is no conflict of interest.


	\bibliographystyle{plain}
	\bibliography{reference}

\end{document}